\documentclass[a4paper,reqno,12pt]{amsart}

\usepackage{todonotes, mathtools, amsmath}
\usepackage[total={16cm,24cm},vmarginratio=4:3,hmarginratio=3:4]{geometry}
\usepackage[utf8]{inputenc}
\usepackage[unicode=true,hidelinks]{hyperref}

\numberwithin{equation}{section}

\usepackage[abbrev,backrefs]{amsrefs}

\usepackage[T1]{fontenc}
\usepackage{lmodern}
\usepackage{microtype}

\usepackage{amssymb}
\newcommand{\defeq}{\coloneqq}

\usepackage{enumerate}

\usepackage{comment}

\usepackage{mathtools}

\newtheorem{proposition}{Proposition}[section]
\newtheorem{theorem}[proposition]{Theorem}
\newtheorem{lemma}[proposition]{Lemma}
\newtheorem{corollary}[proposition]{Corollary}
\theoremstyle{definition}

\theoremstyle{remark}

\newtheorem{example}{Example}[section]

\newcommand{\st}{\: :\:}

\newcommand{\R}{\mathbf{R}}
\newcommand{\D}{\mathcal{D}}

\newcommand{\N}{\mathbf{N}}

\newcommand{\g}{\mathfrak{g}}
\newcommand{\End}{\operatorname{End}}
\newcommand{\Ind}{\mathcal{I}}
\newcommand{\K}{\mathcal{K}}

\DeclareMathOperator{\ad}{ad}

\newcommand{\vertiii}[1]{{\left\vert\kern-0.25ex\left\vert\kern-0.25ex\left\vert #1 
    \right\vert\kern-0.25ex\right\vert\kern-0.25ex\right\vert}}

\title[Limiting inequalities on stratified homogeneous groups]{Limiting Sobolev and Hardy inequalities on stratified homogeneous groups}

\author{Jean Van Schaftingen}
\address{Universit\'e catholique de Louvain\\
Institut de Recherche en Math\'ematique et Physique (IRMP)\\
Chemin du Cyclotron 2\\
1348 Louvain-la-Neuve\\
Belgium}
\email{Jean.VanSchaftingen@uclouvain.be}

\author{Po-Lam Yung}
\address{The Chinese University of Hong Kong \\
Department of Mathematics \\
Shatin \\
Hong Kong}
\email{plyung@math.cuhk.edu.hk}

\address{Australian National University \\
Mathematical Sciences Institute \\
Canberra \\
Australia}
\email{polam.yung@anu.edu.au}

\thanks{Yung was partially supported by the General Research Fund CUHK14313716 from the Hong Kong Research Grant Council, and a Future Fellowship FT200100399 from the Australian Research Council.}

\date{\today}


\begin{document}

\begin{abstract}
We give a sufficient condition for limiting Sobolev and Hardy inequalities to hold on stratified homogeneous groups. In the Euclidean case, this condition reduces to the known cancelling necessary and sufficient condition.
We obtain in particular endpoint Korn--Sobolev and Korn--Hardy inequalities on stratified homogeneous groups.
\end{abstract}

\maketitle

\section{Introduction}

Let \(n, k\) be positive integers with \(n \geq 2\) and let \(V, E\) be finite dimensional inner product spaces over~$\R$. Let \(D\) be the gradient on \(\R^n\) and 
\begin{equation*}
A(D) \colon C^{\infty}(\R^n;V) \to C^{\infty}(\R^n;E)
\end{equation*} 
be a (matrix-valued) homogeneous linear injectively elliptic operator of order \(k\) with constant coefficients  --- that is 
\begin{equation*}
A (D) u = \sum_{\lvert \alpha \rvert = k} A_\alpha \partial^\alpha u
\end{equation*} 
with \(A_\alpha \in \End(V;E)\) where \(A (\xi) \defeq \sum_{\lvert \alpha \rvert = k} A_\alpha \xi^\alpha \in \End(V;E)\) is injective for all \(\xi \ne 0\). For such operators, while it is well-known that 
\begin{equation} \label{eq:CZLp}
  \lVert D^k u \rVert_{L^{p}(\R^n;V)}
  \leq C_p
  \lVert A (D) u \rVert_{L^p (\R^n;E)}
\end{equation}
for all \(u \in C^\infty_c (\R^n, V)\) and \(1 < p < \infty\), the same inequality generally fails when \(p = 1\). Nevertheless, under the above conditions on \(A(D)\), the following statements are known to be equivalent:
\begin{enumerate}[(a)]
\item An endpoint Gagliardo--Nirenberg--Sobolev inequality
\begin{equation}
\label{eq_Iengeexaigheelaigh6uo1ei}
  \lVert D^{k - 1} u \rVert_{L^{n/(n - 1)}(\R^n;V)}
  \leq C
  \lVert A (D) u \rVert_{L^1 (\R^n;E)},
\end{equation}
holds for all \(u \in C^\infty_c (\R^n, V)\);
\item An endpoint Hardy inequality 
\begin{equation}
\label{eq_Iengeexaigheelaigh6uo1ei2}
  \left\| \frac{D^{k - 1} u}{|x|} \right\|_{L^1(\R^n;V)}
  \leq C
  \lVert A (D) u \rVert_{L^1 (\R^n;E)},
\end{equation}
holds for all \(u \in C^\infty_c (\R^n, V)\);
\item  \(A (D)\) is canceling, that is, 
\begin{equation*}
 \bigcap_{\xi \in \R^n \setminus \{0\}} A (\xi)[V] = \{0\}.
\end{equation*}
\end{enumerate}
Note that if \eqref{eq:CZLp} were to hold when \(p = 1\), then \eqref{eq_Iengeexaigheelaigh6uo1ei} and \eqref{eq_Iengeexaigheelaigh6uo1ei2} would be a consequence via Sobolev embedding and the Hardy inequality for the gradient on \(\R^n\) respectively.
The above equivalence was accomplished in a series of works originating from sharp and delicate endpoint estimates by Bourgain and Brezis \citelist{\cite{MR1913720}\cite{MR2038078}\cite{MR1949165}\cite{MR2057026}\cite{VanSchaftingen_2008}\cite{VanSchaftingen_2013}\cite{MR2293957}\cite{MR2075883}\cite{VanSchaftingen_2004_L1}\cite{VanSchaftingen_2004_ARB}\cite{MR3298002}\cite{MR3283556}}.
This canceling condition also plays a role in endpoint \(L^\infty\) estimates \citelist{\cite{MR4087393}\cite{MR4024554}\cite{MR2807409}\cite{MR2578609}}.

In this paper, as in \citelist{\cite{MR3191973}\cite{Chanillo_VanSchaftingen_2009}}, we consider instead of the Euclidean space \(\R^n\) a stratified homogeneous group \(G\).
This means that \(G\) is a connected and simply connected real Lie group, whose Lie algebra \(\g\) is nilpotent  and admits a direct sum decomposition
\begin{equation} \label{eq:Liealg_grading}
\g = \g_1 \oplus \g_2 \oplus \dotsb \oplus \g_r
\end{equation}
where \([\g_i,\g_j] \subset \g_{i+j}\) for all \(i,j\) (and \(\g_{i+j}\) is understood to be zero for \(i+j>r\)); furthermore, \(\g_1\) is assumed to generate \(\g\) as a Lie algebra. The additive group $\R^n$ is the simplest example; all other examples are non-abelian, with the next simplest ones being the Heisenberg groups $\mathbb{H}^n$ that arise in connection to several complex variables and quantum mechanics. The Heisenberg groups are step $2$ groups, meaning that $r$ can be taken to be~$2$ in \eqref{eq:Liealg_grading}; our results, on the other hand, are valid for groups of arbitrarily high steps.

Of particular importance to us is the homogeneous dimension $Q$ of $G$; it is defined as
\begin{equation} \label{eq:Qdef}
Q \defeq \sum_{j = 1}^r j \cdot \dim(\g_j)
\end{equation} 
and arises naturally when one computes the push-forward of the Haar measure $dx$ on $G$ by an automorphic dilation (see \eqref{eq:Q_dilate} below). The homogeneous dimension $Q$ will play the role of $n$, in our generalization of estimates such as \eqref{eq_Iengeexaigheelaigh6uo1ei} to $G$. To describe such a generalization, let us write \(X_1, \dotsc, X_m\) for a basis of \(\g_1\) (in particular, \(m = \dim \g_1\)). Each vector \(X_i\), \(i=1, \dotsc, m\), gives rise to a left-invariant vector field on \(G\), which by abuse of notation we also denote by \(X_i\). Let now \(k\) be a positive integer, and \( \Ind_k \) be the index set \( \{1, \dots, m\}^k \). For \( \gamma = (\gamma_1, \dots, \gamma_k ) \in \Ind_k \), we write \(X_{\gamma}\) for the left-invariant differential operator on $G$ given by
\begin{equation} \label{eq:Xgamma_def}
X_{\gamma} \defeq X_{\gamma_1} \dots X_{\gamma_k}.
\end{equation}
Note that \(X_{\gamma}\) depends on the ordering of the indices within \(\gamma\) if the group \(G\) is not abelian. If \(A^{\gamma} \in \End(V;E)\) for all indices \(\gamma \in \Ind_k\), then
\begin{equation*}
A(D) \defeq \sum_{\gamma \in \Ind_k} A^{\gamma} X_{\gamma},
\end{equation*}
defines a homogeneous left-invariant linear partial differential operator  \(A(D) \colon C^{\infty}(G;V) \to C^{\infty}(G;E)\) of order \(k\) on \(G\) from \(V\) to \(E\) with real coefficients. We are interested in when the inequality 
\begin{equation}
\label{eq_thaX8yohPeilohNgo3lae7ei}
  \lVert D^{k - 1} u \rVert_{L^{Q/(Q - 1)}(G;V)}
  \leq C
  \lVert A (D) u \rVert_{L^1 (G;E)},
\end{equation}
holds for such an operator \(A (D)\), for all \(u \in C^{\infty}_c(G;V)\), where \(D^{k-1} u \defeq (X_{\gamma} u)_{\gamma \in \Ind_{k-1}}\). The inequality \eqref{eq_thaX8yohPeilohNgo3lae7ei} is the natural generalization of \eqref{eq_Iengeexaigheelaigh6uo1ei} to stratified homogeneous groups.

If \(A(D) = D\), then \eqref{eq_thaX8yohPeilohNgo3lae7ei} is an endpoint Sobolev inequality which is known to hold \citelist{\cite{Garofalo_Nhieu_1996}\cite{Franchi_Lu_Wheeden_1995}}.
In the particular case, where \(G\) is a Heisenberg group, Baldi, Franchi and Pansu have proved that \eqref{eq_thaX8yohPeilohNgo3lae7ei} holds when \(A (D)\) is a (first- or second-order) operator of the Rumin complex \citelist{\cite{Baldi_Franchi_2013}\cite{Baldi_Franchi_Pansu_2016}}; their proof relies on the structure of the Rumin complex and on a Bourgain--Brezis duality estimate on stratified homogeneous groups \cite{Chanillo_VanSchaftingen_2009} generalizing the Euclidean results \citelist{\cite{VanSchaftingen_2004_ARB}\cite{VanSchaftingen_2004_L1}\cite{VanSchaftingen_2008}}.

Our main result asserts that for a homogeneous left-invariant linear partial differential operator \(A(D) \colon C^{\infty}(G;V) \to C^{\infty}(G;E)\) of order \(k\) as above, if 
\begin{enumerate}[(i)]
\item \label{condAMH} \(A(D)\) is maximally hypoelliptic, that is, if there exists some \(C > 0\) such that 
\begin{equation*}
\|D^k u\|_{L^2(G;V)} \leq C \|A(D) u\|_{L^2(G;E)}
\end{equation*}
for all \(u \in C^\infty_c (G;V)\), and 
\item \label{condLD} there exists a finite dimensional inner product space \(F\), a positive integer \(\ell\) and some linear homogeneous left-invariant partial differential operator \(L(D) = \sum_{\lambda \in \Ind_{\ell}} B^{\lambda} X_{\lambda}\) of order \(\ell\) on \(G\) from \(E\) to \(F\) (so each \(B^{\lambda} \in \End(E;F)\)) such that the symmetrized operator \(\operatorname{Sym}(L)(D)\) is cocanceling, that is
\begin{equation*}
\bigcap_{\xi \in \R^m} \Big\{ e \in E \st \sum_{\lambda \in \Ind_{\ell}} \xi^\lambda B^\lambda [e] = 0 \Big\} = \{0\},
\end{equation*}
\end{enumerate}
then \eqref{eq_thaX8yohPeilohNgo3lae7ei} holds for all \(u \in C^{\infty}_c(G;V)\) (see Theorem~\ref{thm:can_est} below).
Under the above hypotheses \eqref{condAMH} and \eqref{condLD} on $A(D)$, we also obtain Hardy--Sobolev inequalities of the form 
\begin{equation} \label{eq:Hardy_introduction}
\left\| \frac{D^{k-\ell} u}{\|x\|^{\ell}} \right\|_{L^1(G;V)} \leq C \| A(D) u \|_{L^1(G;E)}
\end{equation}
where \(\|x\|\) is a homogeneous norm on \(G\) and \(\ell \in \{1, \dots, \min\{k,Q-1\}\}\) (see Theorem~\ref{thm:Hardy}(a) below); when \(\ell = 1\) this is the analog of \eqref{eq_Iengeexaigheelaigh6uo1ei2} on \(G\). Similarly, if \(k \ge Q\), the same assumptions on \(A(D)\) implies the \(L^{\infty}\) estimate 
\begin{equation*}
\|D^{k-Q} u\|_{L^{\infty}(G;V)} \leq C \| A(D) u \|_{L^1(G;E)} 
\end{equation*}
which holds for all \(u \in C^{\infty}_c(G;V)\) (see Theorem~\ref{thm:Hardy}(b) below).

When \(G\) is the Euclidean space \(\R^n\), both the ellipticity of \(A(D)\) and the existence of \(L(D)\) as above are known to be necessary for \eqref{eq_thaX8yohPeilohNgo3lae7ei} to hold (see \cite{VanSchaftingen_2013}). 
Another reason why the maximal hypoellipticity condition on \(A(D)\) is a natural one is because under this condition, one has, for all \(u \in C^\infty_c (G;V)\) and all \(1 < p < \infty\), that
\begin{equation} \label{eq:Lpintro}
\|D^k u\|_{L^p(G;V)} \leq C_p \|A(D) u\|_{L^p(G;E)}
\end{equation}
(see e.g.\ Theorem~\ref{thm:kernel_main} below), which implies both
\begin{equation} \label{eq:Lpimproving}
\|D^{k-1} u\|_{L^{Q p/(Q-p)}(G;V)} \leq C_p \|A(D) u\|_{L^p(G;E)}
\end{equation}
for \(1 < p < Q\) via Sobolev embedding, and
\begin{equation} \label{eq:LpHardy}
\left\| \frac{D^{k-\ell} u}{\|x\|^{\ell}} \right\|_{L^p(G;V)} \leq C_p \|A(D) u\|_{L^p(G;E)}
\end{equation}
for \(\ell \in \{1, \dots, \min\{k,Q-1\}\}\), \(1 < p < Q/\ell\) via Hardy's inequality for \(D^{\ell}\).
(Indeed, if \(\ell \in \{1, \dots, Q-1\}\), then for every \(u \in C^{\infty}_c(G;\R)\), one has
\begin{equation*}
|u(x)| \leq C \int_G |D^{\ell} u(x y^{-1})|\frac{1}{\|y\|^{Q-\ell}} dy
\end{equation*}
(this will follow e.g.\ from Theorem~\ref{thm:kernel_main} below). By Ciatti, Cowling and Ricci \cite{MR3336090}*{Theorem A}, one then has
\begin{equation*}
\left\| \frac{u(x)}{\|x\|^{\ell}} \right\|_{L^p(G;\R)} \leq C_p \| D^{\ell} u(x) \|_{L^p(G;\R)}
\end{equation*}
for all \(1 < p < Q/\ell\)
(see also Ruzhansky and Suragan \cite{MR3966452}*{Theorem 7.1.1}).
It follows that if \(u \in C^{\infty}_c(G;V)\), \(k \in \N\), \(\ell \in \{1,\dots,\min\{k,Q-1\}\}\)  and \(1 < p < Q/\ell\), then
\begin{equation*}
\left\| \frac{D^{k-\ell} u}{\|x\|^{\ell}} \right\|_{L^p(G;V)} \leq C_p \|D^k u\|_{L^p(G;V)},
\end{equation*}
which together with \eqref{eq:Lpintro} implies \eqref{eq:LpHardy}.)
The estimates \eqref{eq_thaX8yohPeilohNgo3lae7ei} and \eqref{eq:Hardy_introduction} are then respectively a limiting case of \eqref{eq:Lpimproving} and \eqref{eq:LpHardy} as \( p \to 1^+ \). It may be worth pointing out that the maximal hypoellipticity condition on \(A(D)\) is actually equivalent (under our other assumptions on \(A(D)\)) to a hypoellipticity condition on \(A^t(D) A(D)\); see again Theorem~\ref{thm:kernel_main}.
It is an interesting open question whether the validity of \eqref{eq_thaX8yohPeilohNgo3lae7ei} for all \(u \in C^{\infty}(G;V)\) implies the conditions \eqref{condAMH} and \eqref{condLD} on \(A(D)\) on a general stratified homogeneous group \(G\).

In order to apply our main theorems, given the operator \(A(D)\) one must construct a compatible \(L(D)\) as in condition \eqref{condLD} above. In Proposition~\ref{prop:cocan_preserve} we develop a robust way of doing so that works in many examples of interest. In particular, in Proposition \ref{prop:subelliptic_highergradient} below, we obtain the Gagliardo--Nirenberg--Sobolev type estimate
\begin{equation*}
 \Vert D^{k - 1} u \Vert_{L^{Q/(Q - 1)}(G)}
 \le C \sum_{i = 1}^m \lVert X_i^k u\rVert_{L^1 (G)}
\end{equation*}
for every \(u \in C^\infty_c (G; \R)\), and if \(k \ge Q\), we get an \(L^{\infty}\) estimate
\begin{equation*}
 \Vert D^{k - Q} u \Vert_{L^{\infty }(G)}
 \le C \sum_{i = 1}^m \lVert X_i^k u\rVert_{L^1 (G)}.
\end{equation*}
We also obtain, in Theorem \ref{thm:Korn_Sobolev} below, a Korn--Sobolev type estimate
\begin{equation*}
 \Vert u \Vert_{L^{Q/(Q - 1)}(G)}
 \le C \sum_{i = 1}^m \lVert X_i u_j + X_j u_i\rVert_{L^1 (G)},
\end{equation*}
for every \(u \in C^\infty_c (G; \g_1)\); this Korn--Sobolev inequality was known in the Euclidean case \cite{Strauss_1973}, but it seems to be new in the case of any other stratified homogeneous group. Furthermore, certain limiting Hardy inequalities are obtained for the operators \(u \mapsto (X_1^k u, \dots, X_m^k u)\) and \((u_1,\dots,u_m) \mapsto (X_i u_j + X_j u_i)_{1 \leq i \leq j \leq m}\) in Section~\ref{sect:applications} if \(G \ne \R\) (i.e.\ if \(Q \ge 2\)).

Our proofs of the estimates rely on a representation formulas through a construction of fundamental solution which extends to high-order operators the result of Baldi, Franchi and Tesi \cite{BaldiFranchiTesi2006} (which is the content of Theorem \ref{thm:kernel_prelim}, Corollary \ref{cor:kernel0} and Theorem \ref{thm:kernel_main} below), the Bourgain--Brezis duality estimate on stratified homogeneous groups \cite{Chanillo_VanSchaftingen_2009} (see Lemma~\ref{lem:CVS} and its variant Proposition~\ref{prop:Hardy_dual} below) and various tools to construct \(L (D)\) (as in Proposition~\ref{prop:cocan_preserve} and Lemma~\ref{lem:cocan_open}).

\section{Set up and Preliminaries}

Let $G$ be a stratified homogeneous group and $\g$ be its Lie algebra. The exponential map \(x \in \g \mapsto \exp(x) \in G\) defines a global coordinate chart on \(G\), and allows one to identify \(G\) with \(\g\); we will always use this global coordinate chart to identify \(G\) with a Euclidean space (in particular, the identity element of \(G\) will be denoted by \(0\), and \(G\) inherits the Lebesgue measure \(dx\) from the underlying Euclidean space; note \(dx\) is then the Haar measure on \(G\)). The homogeneous dilation on \(G\) is given by \(x \mapsto \delta_{\lambda} x\), where 
\begin{equation*}
  \delta_{\lambda} x \defeq (\lambda x_1, \lambda^2 x_2, \dotsc, \lambda^r x_r)
\end{equation*} 
if \(\lambda > 0\) and \(x = (x_1, x_2, \dotsc, x_r) \in \g_1 \oplus \g_2 \oplus \dotsb \oplus \g_r\); for any \(\lambda > 0\), \(\delta_{\lambda}\) is an automorphism of \(G\), and the pushforward of \(dx\) by \(\delta_{\lambda}\) is  
\begin{equation} \label{eq:Q_dilate}
  (\delta_{\lambda})_* (dx) = \lambda^{-Q} dx,
\end{equation} 
where \(Q\) is the homogeneous dimension of \(G\) defined in \eqref{eq:Qdef}. For \(d \in \R\), a function \(\phi \colon G \to \R\) is said to be homogeneous of degree \(d\) whenever 
\begin{equation*}
\phi \circ \delta_{\lambda} (x) = \lambda^d \phi(x)
\end{equation*}
for every \(x \in G \setminus \{0\}\) and every \(\lambda > 0\). An example is given by the \emph{homogeneous norm} function on \(G\), defined by
\begin{equation*}
\|x\| \defeq \left( \sum_{j=1}^r |x_j|^{\frac{2r!}{j}} \right)^{\frac{1}{2r!}}
\end{equation*}
if \(x = (x_1, \dotsc, x_r) \in \g_1 \oplus \g_2 \oplus \dotsb \oplus \g_r\); this function is homogeneous of degree \(1\) and \(C^{\infty}\) on \(G \setminus \{0\}\).

Throughout this paper, we will write \(X_1, \dotsc, X_m\) for a basis of \(\g_1\) (in particular, \(m = \text{dim}\, \g_1\)). Each vector \(X_i\), \(i=1, \dotsc, m\), gives rise to a left-invariant vector field on \(G\), which by abuse of notation we also denote by \(X_i\):
\begin{equation*}
X_i \phi(x) \defeq \left. \frac{d}{ds} \right|_{s=0} \phi(x \exp(s X_i))
\end{equation*}
for \(\phi \in C^{\infty}(G;\R)\). The vector field \(X_i\) is left-invariant because it commutes with left translations (i.e.\ \(X_i [\phi(yx)] = (X_i \phi)(yx)\) for any \(y \in G\) and any \(\phi \in C^{\infty}(G;\R)\)). 
We will also denote by 
\begin{equation*}
  D \phi = (X_1 \phi, \dotsc, X_m \phi)
\end{equation*} 
the subelliptic gradient of any function \(\phi \in C^{\infty}(G;\R)\). 

Let \(\N\) denote the set of positive integers and \(\N_0 \defeq \N \cup \{0\}\). For \(k \in \N\), let \(T_k(\g_1)\) be the \(k\)-fold tensor product of \(\g_1\). 
We will write \(\Ind_k\) for the index set \(\{1,\dotsc,m\}^k\), and 
\begin{equation*}
X^{\otimes}_{\gamma} = X_{\gamma_1} \otimes \dots \otimes X_{\gamma_k} \quad \text{if \(\gamma = (\gamma_1, \dotsc, \gamma_k) \in \Ind_k\)},
\end{equation*}
so that \(\{X^{\otimes}_{\gamma}\}_{\gamma \in \Ind_k}\) is a basis of \(T_k(\g_1)\). 
One then has a linear surjection from \(T_k(\g_1)\), to the vector space of all homogeneous left-invariant linear partial differential operator of order \(k\) on \(G\) with real coefficients, given by 
\begin{equation*}
X^{\otimes}_{\gamma} \mapsto X_{\gamma},
\end{equation*} 
where \(X_{\gamma}\) is the differential operator defined in \eqref{eq:Xgamma_def}. The operator \(X_{\gamma}\) is homogeneous of order \(k\), because it sends every homogeneous function in \(C^{\infty}(G;\R)\) to another homogeneous function whose degree is \(k\) lower. 
It is known (see e.g.\ \cites{MR657581,MR2504877}) that if \(k \in \N_0\) and \(f \in C^{\infty}(G;\R)\) satisfies \(X_{\gamma} f = 0\) for all \(\gamma \in \Ind_k\), then \(f\) is a polynomial of \(x\) on \(G\) of non-isotropic degree less then \(k\), i.e.\ \(f(x)\) is a linear combination of 
\begin{equation*}
    x^{\alpha} \defeq x_1^{\alpha_1} \dotsm x_r^{\alpha_r}
\end{equation*} 
where \(\alpha = (\alpha_1, \dotsc, \alpha_r) \in \N_0^{\dim \g_1} \times \cdots \times \N_0^{\dim \g_r} = \N_0^{\dim \g}\) satisfies \(\|\alpha\| < k\); henceforth
\begin{equation*}
\|\alpha\|\defeq\sum_{j=1}^r j |\alpha_j|.
\end{equation*}

Later we will also need right-invariant versions of \(X_1, \dotsc, X_m\). They are defined by 
\begin{equation*} 
    X^R_i \phi(x) = \left. \frac{d}{ds} \right|_{s=0} \phi(\exp(sX_i) x)
\end{equation*} 
for \(i = 1,\dotsc, m\)  and \(\phi \in C^{\infty}(G;\R)\); equivalently, 
\begin{equation*}
    X^R_i \phi = -\widetilde{ X_i \tilde{\phi}}
\end{equation*} 
where \(\tilde{\phi}(x) \defeq \phi(x^{-1})\). We write \(D^R \phi = (X^R_1 \phi, \dotsc, X^R_m \phi)\), and write 
\begin{equation*}
  X^R_{\gamma} \defeq X^R_{\gamma_1} \dotsm X^R_{\gamma_k}
\end{equation*} 
if \(\gamma = (\gamma_1, \dotsc, \gamma_k) \in \Ind_k\); it is also homogeneous of order \(k\). If \(\phi\) is a \(C^{\infty}\) function on \(G\) taking values in a real vector space, then for each \(k \in \N\) and \(\gamma \in \Ind_k\), \(X_{\gamma} \phi\) and \(X^R_{\gamma} \phi\) are defined componentwise.

Derivatives with respect to the coordinates \(x\) on \(G\) will be denoted by \(\partial_x^{\alpha}\) where \(\alpha \in \N_0^{\dim \g}\) is a multiindex; they are typically neither left nor right-invariant and we will only use them in local considerations.

Let \(\D(G) \defeq C^{\infty}_c(G;\R)\) denote the space of test functions on \(G\). Upon identifying \(G\) with the underlying Euclidean space, \(\D(G)\) is endowed with a \(\mathcal{LF}\)-topology (strict inductive limit of Fr\'echet spaces, see e.g.\ Rudin \cite{Rudin}, Grubb \cite{Grubb}), which we recall as follows. 
For any compact subset \(K\) of \(G\), let \(\D(K)\) be the set of all \(\phi \in \D(G)\) with support contained in \(K\). \(\D(K)\) is equipped with the usual Fr\'{e}chet (i.e.\ locally convex, metrizable and complete) topology, via a countable family of separating seminorms \(\{\|\cdot\|_{C^n(K)} \colon n \in \N_0\}\) where 
\begin{equation*}
\|\phi\|_{C^n(K)} \defeq \sum_{|\alpha| \leq n} \sup_{x \in K} |\partial_x^{\alpha} \phi(x)|
\end{equation*}
for \(\phi \in \D(K)\); here \(|\alpha|\) is the length of the multiindex \(\alpha \in \N^{\dim \g}\) defined by \(|\alpha| \defeq \sum_{i=1}^{\dim \g} \alpha_i\) if \(\alpha = (\alpha_1, \dotsc, \alpha_{\dim \g})\). 
A sequence \((\phi_i)_{i \in \N}\) of \(\D(K)\) converges in the topology of \(\D(K)\) to some \(\phi \in \D(K)\), if and only if \begin{equation*}\lim_{j \to \infty} \|\phi_j - \phi\|_{C^n(K)} = 0\end{equation*} for all \(n \in \N_0\); indeed, if \(U_{n,\varepsilon} \defeq \{\phi \in \D(K) \colon \|\phi\|_{C^n(K)} < \varepsilon\}\) then \(\{U_{n,\varepsilon} \colon n \in \N_0, \varepsilon > 0\}\) is a local basis for the system of neighborhoods at \(0 \in \D(K)\) (cf.\ \cite{Grubb}*{Remark B.6}).
From this it is not difficult to show that a linear functional \(u \colon \D(K) \to \R\) is continuous, if and only if there exists \(n \in \N_0\) and \(c > 0\) such that
\begin{equation*}
|u(\phi)| \leq c \|\phi\|_{C^n(K)}
\end{equation*}
for all \(\phi \in \D(K)\); similarly, a linear operator \(T \colon \D(K) \to \D(K)\) is continuous, if and only if for any \(n \in \N_0\), there exists \(n' \in \N_0\) and \(C > 0\) such that
\begin{equation*}
\|T \phi\|_{C^n(K)} \leq C \|\phi\|_{C^{n'}(K)}
\end{equation*} 
for all \(\phi \in \D(K)\) (cf.\ \cite{Grubb}*{Lemma B.7}).
If \((K_i)_{i \in \N}\) is an increasing sequence of compact subsets of \(G\) that exhaust \(G\), then \(\D(K_i) \hookrightarrow \D(K_{i'})\) is continuous whenever \(i < i'\), and \begin{equation*}
    \D(G) = \bigcup_{i \in \N} \D(K_i)
\end{equation*} 
is endowed with the direct limit topology in the category of locally convex topological vector spaces, i.e.\ the finest (a.k.a.\ strongest) locally convex topology on \(\D(G)\) so that the inclusions \(\D(K_i) \hookrightarrow \D(G)\) is continuous for every \(i \in \N\). 
This topology is independent of the choice of the compact exhaustion of \(G\) (and finer than the topology on \(\D(G)\) inherited from the Fr\'{e}chet topology of \(C^{\infty}(G)\), which we will never use).
With the topology of \(\D(G)\) in place, the space \(\D'(G)\) of all real distributions on \(G\) is then defined as the space of all continuous linear functionals from \(\D(G)\) to \(\R\); as usual \(\D'(G)\) is equipped with the weak* topology. 
More concretely, a sequence \((\phi_j)_{j \in \N}\) of \(\D(G)\) converges in the topology of \(\D(G)\) to some \(\phi \in \D(G)\), if and only if there exists \(i \in \N\) so that \(\phi, \phi_j\) are supported in \(K_i\) for every \(j \in \N\), and \(\phi_j\) converges to \(\phi\) in the topology of \(\D(K_i)\). 
A linear functional \(u \colon \D(G) \to \R\) is in \(\D'(G)\), if and only if for any \(i \in \N\), \(u\) restricts to a continuous linear functional on \(K_i\). 
A sequence \((u_j)_{j \in \N}\) of \(\D'(G)\) converges in the topology of \(\D'(G)\) to some \(u \in \D'(G)\), if and only if \(u_j(\phi) \to u(\phi)\) for every \(\phi \in \D(G)\). Finally, a linear operator \(T \colon \D(G) \to \D(G)\) is continuous, if and only if its restriction \(T \colon \D(K_i) \to \D(G)\) is continuous for every \(i \in \N\) (this will be the case, for instance, if \(T\) restricts to a continuous operator from \(\D(K_i)\) to \(\D(K_i)\) for every \(i \in \N\)); 
it then induces by duality a continuous linear operator \(T^* \colon \D'(G) \to \D'(G)\), via \((T^*u)(\phi) = u(T \phi)\) for all \(\phi \in \D(G)\).

If \(V\) is a finite dimensional inner product space over \(\R\), we write \(\D(G;V) \defeq \D(G) \otimes V\) (the space of \(V\)--valued test functions), and \(\D'(G;V) \defeq \D'(G) \otimes V\) (the space of \(V\)--valued distributions); they are isomorphic to products of finitely many copies of \(\D(G)\) and \(\D'(G)\) respectively, and hence are endowed with the corresponding topologies. The natural pairing between \(\D(G;V)\) and \(\D'(G;V)\) will be denoted by \(\langle \cdot, \cdot \rangle_{V,G}\). Similarly we define \(C^{\infty}(G;V) \defeq C^{\infty}(G) \otimes V\), \(L^p(G;V) \defeq L^p(G) \otimes V\) for \(1 \leq p \leq \infty\), and \(H^s(G;V) \defeq H^s(G) \otimes V\) for \(s \in \R\) with the natural \(H^s(G;V)\) inner product; also, \(\D(K;V) \defeq \D(K) \otimes V\) if \(K\) is a compact subset of \(G\), and \(\D(K;V)\) is endowed with the natural topology from \(\D(K)\). 

A distribution \(u \in \D'(G;V)\) is said to be homogeneous of degree \(d \in \R\), whenever
\begin{equation*}
u \circ \delta_s = s^d u\end{equation*} for all \(s > 0\); here \(u \circ \delta_s\) is defined by
\begin{equation*}
\langle u \circ \delta_s, \phi \rangle_{V,G} = \langle u, s^{-Q} \phi \circ \delta_{s^{-1}} \rangle_{V,G} 
\end{equation*}
for all \(\phi \in \D(G;V)\) and all \(s > 0\). For instance, from now on, let \(\delta \in \D'(G)\) be the Dirac delta at \(0\). Then for any multiindex \(\alpha \in \N_0^{\dim \g}\), the distribution \(\partial_x^{\alpha} \delta \otimes v_{\alpha} \in \D'(G;V)\) is  homogeneous of degree \(-Q-\|\alpha\|\) where \(v_{\alpha}\) is any vector in \(V\). 
Later we will need the fact that if \(u \in \D'(G;V)\) is supported at \(\{0\}\), then \(u\) is a finite linear combination of such \(\partial_x^{\alpha} \delta \otimes v_{\alpha}\); in particular, if in addition \(u\) is homogeneous of some degree \(d > -Q\), then \(u = 0\). 

Let \(V\), \(E\) be finite dimensional inner product spaces over \(\R\) and \(k \in \N\). We will always identify \(V\) with its dual space \(V^*\) via the inner product, and similarly for \(E\), so that \(\End(V;E) = E \otimes V\). The space \(\End(V;E) \otimes T_k(\g_1)\) can be viewed as either the space of all linear maps from \(V\) to \(E\) whose coefficients are \(k\) tensors on \(\g_1\), or the space of all degree \(k\) tensors on \(\g_1\) whose coefficients are linear maps from \(V\) to \(E\); we usually adopt the second point of view. 
The formal adjoint \(A^t(D)\) of \(A(D)\) is an operator \(A^t(D) \colon C^{\infty}(G;E) \to C^{\infty}(G;V)\), given by
\begin{equation*}
A^t(D) \defeq \sum_{\gamma \in \Ind_k} (A^{\gamma})^t X_{\gamma}^t
\end{equation*}
where \((A^{\gamma})^t\in \End(E;V)\) is the transpose of \(A^{\gamma}\), and 
\begin{equation*}
(X_{\gamma})^t = (-1)^k X_{\gamma_k} \dots X_{\gamma_1}
\end{equation*} 
if \(\gamma = (\gamma_1, \dotsc, \gamma_k) \in \Ind_k\). Since \(A^t(D) \colon \D(G;V) \to \D(G;E)\) is continuous (in fact, \(A^t(D) \colon \D(K;V) \to D(K;E)\) is continuous for every compact subset \(K\) of \(G\)), we may extend \(A(D)\) as a continuous operator \(A(D) \colon \D'(G;V) \to \D'(G;E)\) via
\begin{equation*}
\langle A(D) u, \phi \rangle_{E,G} \defeq \langle u, A^t(D) \phi \rangle_{V,G}
\end{equation*}
if \(u \in \D'(G;V)\) and \(\phi \in \D(G;E)\). The operator \(A(D)\) is said to be \emph{hypoelliptic}, whenever for any \(u \in \D'(G;V)\) and any open set \(\Omega \subset G\), \(A(D) u \in C^{\infty}(\Omega;E)\) implies \(u \in C^{\infty}(\Omega;V)\). Similarly, we may extend \(A^t(D)\) as a continuous operator \(A^t(D) \colon \D'(G;E) \to \D'(G;V)\), and define what it means for \(A^t (D)\) to be hypoelliptic. 

Let \(k \in \N\) and \(A \in \End(V;E) \otimes T_k(\g_1)\). Given \(\psi \in \D(G;V)\) we would eventually like to achieve two goals:
\begin{enumerate}[(a)]
\item \label{firstgoal} solve the equation \(A^t(D) u = \psi\), and
\item \label{secondgoal} recover \(\psi\) from \(A(D)\psi\),
\end{enumerate}
both under appropriate conditions on \(A(D)\); this will be accomplished in Theorem~\ref{thm:kernel_main} below. Due to the invariance of \(A(D)\) under left translations, a useful tool in achieving these goals is convolutions on \(G\), whose basics we review next.

First, if \(\psi, \phi \in \D(G;\R)\) then the convolution of \(\psi\) with \(\phi\) is defined to be
\begin{equation*}
\psi*\phi(x) = \int_G \psi(y) \phi(y^{-1} x) \, dy 
= \int_G \psi(x y^{-1}) \phi(y) \, dy \in \D(G).
\end{equation*}
Note that if \(G\) is not abelian then generally \(\psi*\phi \ne \phi*\psi\). It is easy to check that if \(1 \leq i \leq m\), then
\begin{equation*}
X_i (\psi*\phi) = \psi * (X_i \phi), \quad X^R_i (\psi*\phi) = (X^R_i \psi)*\phi, \quad \text{and} \quad (X_i \psi) * \phi = \psi * (X^R_i \phi)
\end{equation*}
for all \(\psi, \phi \in \D(G)\). If \(\psi \in \D(G)\) and \(u \in \D'(G)\), then \(\psi*u \in \D'(G)\) is defined by duality as 
\begin{equation*}
(\psi*u)(\phi) = u(\tilde{\psi}*\phi)
\end{equation*}
where \(\tilde{\psi}(x) \defeq \psi(x^{-1})\). We still have 
\begin{equation*}
X_i (\psi*u) = \psi * (X_i u), \quad X^R_i (\psi*u) = (X^R_i \psi)*u, \quad \text{and} \quad (X_i \psi) * u = \psi * (X^R_i u)
\end{equation*}
for \(\psi \in \D(G)\) and \(u \in \D'(G)\).

Now let \(\psi \in \D(G;V)\) and \(\K \in \D'(G;\End(V;E)) = \D'(G;E \otimes V)\). Let \(\{v_i\}_{i=1}^{\dim V}\) and \(\{e_j\}_{j=1}^{\dim E}\) be orthonormal bases of \(V\) and \(E\) respectively, and write 
\begin{equation*}
\psi(x) = \sum_{i=1}^{\dim V} \psi_i(x) v_i \quad \text{and} \quad \K(x) = \sum_{i=1}^{\dim V} \sum_{j=1}^{\dim E} \K_{ij}(x) e_j \otimes v_i.
\end{equation*} 
Then the convolution \(\psi * \K \in \D'(G;E)\) is defined by
\begin{equation*}
\psi * \K = \sum_{j=1}^{\dim E} \sum_{i=1}^{\dim V} [\psi_i * \K_{ij}] e_j.
\end{equation*}
It satisfies
\begin{equation*}
\langle \psi*\K, \Phi \rangle_{E,G} = \langle \K, \tilde{\psi}*\Phi \rangle_{E \otimes V,G}
\end{equation*}
for all \(\Phi \in \D(G;E)\) where 
\begin{equation*}
\tilde{\psi}*\Phi(x) \defeq \sum_{i=1}^{\dim V} \sum_{j=1}^{\dim E} [\tilde{\psi}_i * \Phi_j(x)] e_j \otimes v_i \in \D(G;E \otimes V)
\end{equation*}
if \(\Phi = \sum_{j=1}^{\dim E} \Phi_j(x) e_j\). 

To proceed further, since \(\End(V;V) = V \otimes V\) and \(\End(V;E) = E \otimes V\), the operator \(A(D) \colon \D(G;V) \to \D(G;E)\) extends via tensor product with \(V\) as a continuous operator \(A(D) \colon \D(G;\End(V;V)) \to \D(G;\End(V;E))\), so that \(A(D)(\psi \otimes v) \defeq (A(D)\psi) \otimes v\) if \(\psi \in \D(G;V)\) and \(v \in V\); in particular, if \(\psi, \phi \in \D(G;V)\), \(\psi(x) = \sum_{i=1}^{\dim V} \psi_i(x) v_i\), \(\phi(x) = \sum_{j=1}^{\dim V} \phi_j(x) v_j\) with 
\begin{equation*}
\tilde{\psi}*\phi \defeq \sum_{i=1}^{\dim V} \sum_{j=1}^{\dim V} [\tilde{\psi}_i*\phi_j] v_j \otimes v_i,
\end{equation*}
then
\begin{equation*}
A(D)[\tilde{\psi} * \phi] = \tilde{\psi} * A(D) \phi.
\end{equation*}
We may then define \(A^t(D) \colon \D'(G;\End(V;E)) \to \D'(G;\End(V;V))\), so that if \(u \in \D'(G;E)\) and \(v \in V\), then \(A^t(D)(u \otimes v) \defeq (A^t(D)u) \otimes v\). If \(\K \in \D'(G;\End(V;E))\) is such that 
\begin{equation} \label{eq:K_solves_AtD}
A^t(D) \K = \delta \otimes I \quad \text{on \(G\)}
\end{equation} 
where \(\delta\) is the delta function at \(0 \in G\) and \(I\) is the identity map on \(V\), then 
\begin{equation*}
A^t(D) [\psi*\K] = \psi 
\end{equation*}
for all \(\psi \in \D(G;V)\), because then for every \(\phi \in \D(G;V)\) we have
\begin{equation*}
\begin{split}
\langle \psi*\K, A(D) \phi \rangle_{E,G} &= \langle \K, \tilde{\psi}*A(D) \phi \rangle_{E \otimes V,G} \\
&= \langle \K, A(D)[\tilde{\psi} * \phi] \rangle_{E \otimes V,G} \\
& = \langle \delta \otimes I, \tilde{\psi} * \phi \rangle_{V \otimes V,G} \\
&= \langle \psi, \phi \rangle_{V,G}.
\end{split}
\end{equation*}
Hence our goal (\ref{firstgoal}) above reduces to the construction of \(\K \in \D'(G;\End(V;E))\) so that \eqref{eq:K_solves_AtD} is satisfied. Indeed, under suitable hypothesis on \(A\), we will construct some \(\K^{\circ} \in \D'(G;\End(V;V))\) for which \(A^t(D) A(D) \K^{\circ} = \delta \otimes I\), and set \(\K \defeq A(D) \K^{\circ}\). The distribution \(\K^{\circ}\) will also allow us to construct \(\tilde{\K} \in \D'(G;\End(E;V))\) such that \(\psi - (A(D)\psi)*\tilde{\K}\) is in the nullspace of \(A^t(D) A(D)\) for any \(\psi \in \D(G;V)\). If further $k < Q$, we will see that \(\psi - (A(D)\psi)*\K = 0\), which achieves our goal (\ref{secondgoal}) above.

Finally, for \(\ell \in \N\), a distribution \(\K \in \D'(G;\End(V;E))\) is said to be a kernel of type \(\ell\), if \(\K\) is homogeneous of degree \(\ell-Q\) and \(C^{\infty}\) on \(G \setminus \{0\}\). Since $\ell > 0$, necessarily such \(\K\) are given by integration against an \(\End(V;E)\)--valued function, that is homogeneous of degree \(\ell-Q\) and \(C^{\infty}\) on \(G \setminus \{0\}\). Furthermore, for every \(\gamma \in \Ind_{\ell}\), the map 
\begin{equation*}
\psi \mapsto \psi * X_{\gamma} \K,
\end{equation*} 
initially defined for \(\psi \in D(G;V)\), extends to a bounded linear operator from \(L^p(G;V)\) to \(L^p(G;E)\) for every \(1 < p < \infty\); this is a consequence of Calder\'{o}n-Zygmund theory on \(G\), and follows, for instance, from Theorem~4 in \cite{Stein}*{Chapter XIII, Section 5.3}. It should be noted that so far we refrained from defining kernels of type \(\ell\) when \(\ell \leq 0\); in that case, the correct definition of a kernel of type \(\ell\) should involve certain additional cancellation conditions, which we will not go into.

\section{A left inverse to \texorpdfstring{\(A(D)\)}{A(D)} and a right inverse to \texorpdfstring{\(A^t(D)\)}{Aᵗ(D)}}

Our first result provides a first sufficient condition on \(A \in \End(V;E) \otimes T_k(\g_1)\) under which the equation \(A^t(D) u = \psi\) may be solved for every \(\psi \in \D(G;V)\). 

\begin{theorem} \label{thm:kernel_prelim}
Let \(k \in \N\), \(V, E\) be finite dimensional inner product spaces over \(\R\), and \(A \in \End(V;E) \otimes T_k(\g_1)\). Suppose \(A(D) \colon \D'(G;V) \to \D'(G;E)\) and \(A^t (D) \colon \D'(G;E) \to \D'(G;V)\) are both hypoelliptic. Then there exists \(\K \in \D'(G;\End(V;E))\) such that 
\begin{equation*}
A^t(D) \K = \delta \otimes I \quad \text{on \(G\)}
\end{equation*}
where \(\delta\) is the delta function at \(0 \in G\) and \(I\) is the identity map on \(V\).
Furthermore, if \(k < Q\), then \(\K\) is a kernel of type \(k\); on the other hand, if \(k \geq Q\), then there exists an \(\End(V;E)\)--valued \(C^{\infty}\) function \(\K_{\infty}\) on \(G \setminus \{0\}\) that is homogeneous of degree \(k-Q\), and a homogeneous \(\End(V;E)\)--valued polynomial \(P\) of degree \(k-Q\), such that 
\begin{equation*}
\K = \K_{\infty}(x) + P(x) \log \|x\| \quad \text{on \(G\)}
\end{equation*}
(in the sense that \(\K\) is distribution given by integration against the right hand side). 
\end{theorem}

The above theorem is essentially contained in Theorem~3.1(i,ii) of Baldi, Franchi and Tesi~\cite{BaldiFranchiTesi2009} if we impose the additional hypotheses that \(V = E\), \(A(D) = A^t(D)\) and \(k \leq Q\). Their proof, whose essence can be found in \cite{BaldiFranchiTesi2008} and \cite{BaldiFranchiTesi2006} and has its roots in Folland \cite{MR0494315}*{Theorem~2.1}, can be extended with relatively little difficulty to cover our slightly more general case in Theorem~\ref{thm:kernel_prelim}, where we allow possibly \(V \ne E\), \(A(D) \ne A^t(D)\) and \(k > Q\). We will provide a proof of Theorem~\ref{thm:kernel_prelim} in an Appendix. Roughly speaking, the hypoellipiticity assumption on \(A(D)\) allows for the construction of a local right inverse of \(A^t(D)\), while the hypoellipticity assumption on \(A^t(D)\) allows one to rescale the above local right inverse to \(A^t(D)\) to a global right inverse \(\K\).

In practice, the application of Theorem~\ref{thm:kernel_prelim} is limited by the fact that it only applies when both \(A(D)\) and \(A^t(D)\) are hypoelliptic. There are natural situations where this assumption is not satisfied; for instance, when \(G = \R^n\), \(V = \R\), \(E = \R^n\) and \(A(D) = D\) the usual gradient on \(\R^n\), then \(A(D)\) is (hypo)elliptic, but \(A^t(D) = -\textrm{div}\) is not hypoelliptic when \(n \ge 2\). Still we expect the equation \(-\textrm{div}\, u = \psi\) to be solvable for all \(\psi \in \D(G;V)\): one would first solve \(-\Delta v = \psi\) and let \(u = D v\), and this works because \(-\Delta = A^t(D) A(D)\). This argument gives us the following more robust sufficient condition under which the equation \(A^t(D) u = \psi\) is solvable for all \(\psi \in \D(G;V)\):

\begin{corollary} \label{cor:kernel0}
Let \(k \in \N\), \(V, E\) be finite dimensional inner product spaces over \(\R\), and \(A \in \End(V;E) \otimes T_k(\g_1)\). Suppose \(A^t(D) A(D) \colon \D'(G;V) \to \D'(G;V)\) is hypoelliptic. Then the conclusions of Theorem~\ref{thm:kernel_prelim} continue to hold.
\end{corollary}

\begin{proof}
Indeed, since \(A^t(D) A(D)\) is its own formal adjoint, Theorem~\ref{thm:kernel_prelim} implies the existence of a fundamental solution \(\K^{\circ} \in \D'(G;\End(V;V))\) so that 
\begin{equation*}
A^t(D) A(D) \K^{\circ} = \delta \otimes I \quad \text{on \(G\)}.
\end{equation*} 
It remains to observe that \(\K \defeq A(D) \K^{\circ} \in \D'(G;\End(V;E))\) satisfies the conclusions of Theorem~\ref{thm:kernel_prelim}, by considering separately the cases \(0 < k < Q/2\), \(Q/2 \leq k < Q\), and \(k > Q\). When \(0 < k < Q/2\), then \(K^{\circ}\) is a kernel of type $2k$, so \(A(D) K^{\circ}\) is a kernel of type \(2k-k = k\). When \(Q/2 \leq k < Q\), then \(\K^{\circ} = \K^{\circ}_{\text{homo}} +  P^{\circ}(x) \log \|x\|\) where \(K^{\circ}_{\text{homo}} \in \D'(G;\End(V;E))\) is homogeneous of degree \(2k-Q\), and \(P^{\circ}(x)\) is a homogeneous \(\End(V;V)\)--valued polynomial of degree \(2k-Q < k\); but \(\K = A(D) \K^{\circ}\) involves \(k\) homogeneous derivatives of \(\K^{\circ}_{\text{homo}} + P^{\circ}(x) \log \|x\|\), so at least one of the \(k\) derivatives must hit the \(\log \|x\|\) factor in the second term in order for it to make a non-zero contribution, transforming the second term into one that is homogeneous of degree \(k-Q\). This shows that \(\K\) is a kernel of type \(k\) in this case. Finally, when \(k \geq Q\), we still have \(\K^{\circ} = \K^{\circ}_{\text{homo}} +  P^{\circ}(x) \log \|x\|\), but now that \(k \geq Q\), the degree \(2k-Q\) of the polynomial \(P^{\circ}\) is at least \(k\). As a result, the best one can get out of the above argument is that \(K = \K_{\infty}(x) + P(x) \log \|x\|\) for some \(\End(V;E)\)--valued \(C^{\infty}\) function \(\K_{\infty}\) on \(G \setminus \{0\}\) that is homogeneous of degree \(k-Q\), and some homogeneous \(\End(V;E)\)--valued polynomial \(P\) of degree \(k-Q\).
\end{proof}

To proceed further, let \(A \in \End(V;E) \otimes T_k(\g_1)\) for some \(k \in \N\). Then the following conditions are equivalent:
\begin{enumerate}[(a)]
\item There exists an open set \(\Omega \subset G\) containing \(0\), and some \(C_{\Omega} > 0\), such that 
\begin{equation} \label{eq:A_subellip_assump}
\|D^k u\|_{L^2(G;V)} \leq C_{\Omega} \left[ \|A(D) u\|_{L^2(G;E)} + \|u\|_{L^2(G;V)} \right] 
\end{equation}
for all \(u \in \D(\Omega,V)\) (henceforth \(D^k\) is a shorthand for \(X_{\gamma}\) for any \(\gamma \in \Ind_k\)). 
\item There exists some \(C > 0\) such that 
\begin{equation} \label{eq:A_subellip_assump_homo}
\|D^k u\|_{L^2(G;V)} \leq C \|A(D) u\|_{L^2(G;E)}
\end{equation}
for all \(u \in \D(G,V)\).
\end{enumerate}
In fact, the first condition implies the second condition by scaling and homogeneity of \(A(D)\), and the second condition clearly implies the first. If \(A(D)\) satisfies either of these conditions, then \(A(D)\) is said to be \emph{maximally hypoelliptic}. It implies that 
\begin{equation} \label{eq:A_max_sub}
\|D^k u\|_{L^2(G;V)} \leq C_{\Omega} \left[ \langle A^t(D) A(D) u, u \rangle_{V,G} + \|u\|_{L^2(G;V)} \right] 
\end{equation}
for all \(u \in \D(\Omega,V)\), which is sometimes described as \(A^t(D) A(D)\) being \emph{maximally hypoelliptic of type 2} on \(\Omega\) (see Street \cite{Street}*{Section 2.4.1}), or \(A^t(D) A(D)\) being \emph{maximally subelliptic} on \(\Omega\) (see Baldi, Franchi, Tesi \cite{BaldiFranchiTesi2009}*{Theorem 4.1}). Again we may scale away the lower order term on the right-hand side of \eqref{eq:A_max_sub}, and obtain
\begin{equation} \label{eq:A_max_sub_homo}
\|D^k u\|_{L^2(G;V)} \leq C \langle A^t(D) A(D) u, u \rangle_{V,G}
\end{equation}
for all \(u \in \D(G;V)\). 
It is known that via microlocalization techniques, \eqref{eq:A_max_sub_homo} implies that \(A^t(D) A(D) \colon \D'(G;V) \to \D'(G;V)\) is hypoelliptic on \(G\) (see e.g.\ \cite{Street}*{Theorem 2.4.11}, \cite{BaldiFranchiTesi2009}*{Theorem 4.1}). To summarize, maximal hypoellipticity of \(A(D)\) implies the hypoellipticity of \(A^t(D) A(D)\). We will see that the converse also holds, and in fact we have the following theorem.

\begin{theorem} \label{thm:kernel_main}
Let \(k \in \N\),  \(V\), \(E\) be finite dimensional inner product spaces over \(\R\),  and \(A \in \End(V;E) \otimes T_k(\g_1)\). Then \(A(D) \colon \D'(G;V) \to \D'(G;E)\) is maximally hypoelliptic, if and only if \(A^t(D) A(D) \colon \D'(G;V) \to \D'(G;V)\) is hypoelliptic. Furthermore, under either of these conditions, the following conclusions hold:
\begin{enumerate}[(i)]
\item \label{cor:i} There exists \(\K \in \D'(G;\End(V;E))\) such that for every \(\psi \in \D(G;V)\) we have
\begin{equation} \label{conclusion1}
\psi = A^t(D) [\psi*\K].
\end{equation}
Furthermore, for any \(\ell \in \N_0\) and any \(\gamma \in \Ind_{\ell}\), there exists \(\K_{\gamma} \in \D'(G;\End(V;E))\) such that 
\begin{equation} \label{conclusion2}
(X_{\gamma})^t \psi = A^t(D) [\psi*\K_{\gamma}]
\end{equation}
for every \(\psi \in \D(G;V)\). If \(\ell > k-Q\), then \(\K_{\gamma}\) is a homogeneous \(\End(V;E)\)--valued distribution of (negative) degree \(k-\ell-Q\), and \(\K_{\gamma}\) agrees with a \(\End(V;E)\)--valued \(C^{\infty}\) function on \(G \setminus \{0\}\), so if further \(\ell < k\), then \(\K_{\gamma}\) is a kernel of type \(k-\ell\).
\item \label{cor:iii} There exists \(\tilde{\K} \in \D'(G;\End(E;V))\) such that for every \(\psi \in \D(G;V)\) we have
\begin{equation}\label{eq_quohohw4she8Aphaedie9aja}
\psi(x) - (A(D)\psi)*\tilde{\K}(x) = 
\begin{cases}
0 &\quad \text{if \(k < Q\)} \\
p(x) &\quad \text{if \(k \geq Q\)}
\end{cases}
\end{equation}
where in the second case \(p(x)\) is a polynomial of \(x\) whose non-isotropic degree is at most  \(k-Q\). 
Furthermore, for any \(\ell \in \N_0\) with \(\ell > k-Q\) and any \(\gamma \in \Ind_{\ell}\), there exists \(\tilde{\K}_{\gamma} \in \D'(G;\End(E;V))\), homogeneous of (negative) degree \(k-\ell-Q\) and agrees with an \(\End(E;V)\)--valued \(C^{\infty}\) function on \(G \setminus \{0\}\), such that
\begin{equation} \label{conclusion4}
X_{\gamma} \psi = (A(D)\psi)*\tilde{\K}_{\gamma} 
\end{equation}
for every \(\psi \in \D(G;V)\). So if further \(\ell < k\), then \(\tilde{\K}_{\gamma}\) is a kernel of type \(k-\ell\). Finally, for \(1 < p < \infty\) and \(\psi \in \D(G;V)\),
\begin{equation} \label{conclusion5}
\sum_{\gamma \in \Ind_k} \|X_{\gamma} \psi\|_{L^p(G;V)} \leq C \|A(D) \psi\|_{L^p(G;E)}.
\end{equation}
\end{enumerate} 
\end{theorem}

This improves upon our Theorem~\ref{thm:kernel_prelim} earlier, because if \(A(D)\) and \(A^t(D)\) are both hypoelliptic, then so is \(A^t(D) A(D)\) and Theorem~\ref{thm:kernel_main} applies. It also strengthens Theorem~3.1 of Baldi, Franchi and Tesi \cite{BaldiFranchiTesi2009}, because if \(V = E\) and \(A(D) = A^t(D)\) is hypoelliptic, then again \(A^t(D) A(D)\) is hypoelliptic, and the conclusions of Theorem~\ref{thm:kernel_main} imply the conclusions (i)--(iv) of \cite{BaldiFranchiTesi2009}*{Theorem 3.1}. Again we note that we impose no upper bound assumption on the order \(k\) of the operator \(A(D)\) in our Theorem~\ref{thm:kernel_main}.

Given what we have done so far, the main remaining difficulty in proving Theorem~\ref{thm:kernel_main} is in the proof of conclusion \eqref{cor:iii}. To that end, a useful tool is a clever Liouville--type theorem from Baldi, Franchi and Tesi \cite{BaldiFranchiTesi2009}, as we will see below.

\begin{proof}[Proof of Theorem~\ref{thm:kernel_main}]
Let \(A \in \End(V;E) \otimes T_k(\g_1)\) for some \(k \in \N\). 

We have already seen that if \(A(D) \colon \D'(G;V) \to \D'(G;E)\) is maximally hypoelliptic, then \(A^t(D) A(D) \colon \D'(G;V) \to \D'(G;V)\) is hypoelliptic. 

So now suppose \(A^t(D) A(D) \colon \D'(G;V) \to \D'(G;V)\) is hypoelliptic, so that Corollary~\ref{cor:kernel0} applies. We will prove the conclusions \eqref{cor:i} and \eqref{cor:iii}. The last conclusion of \eqref{cor:iii}, or more specifically \eqref{conclusion5} applied with \(p=2\), shows that \(A(D) \colon \D'(G;V) \to \D'(G;E)\) is maximally hypoelliptic, and this will complete our proof of the present theorem.

In order to prove \eqref{cor:i}, from Corollary~\ref{cor:kernel0}, we obtain some \(\K \in \D'(G;\End(V;E))\) so that 
\begin{equation*}
A^t(D) \K = \delta \otimes I.
\end{equation*}
As observed before, this shows \(\psi = A^t(D)[\psi*\K]\) for all \(\psi \in \D(G;V)\), which is \eqref{conclusion1}. 

Next, let \(\ell \in \N_0\), \(\gamma \in \Ind_{\ell}\). We apply the above identity to \((X_{\gamma})^t \psi\) in place of \(\psi\). Then
\begin{equation*}
(X_{\gamma})^t \psi = A^t(D) [[(X_{\gamma})^t \psi]* \K] = A^t(D) [\psi * \K_{\gamma}]
\end{equation*}
where 
\begin{equation*}
\K_{\gamma} \defeq (-1)^{\ell} X^R_{\gamma} \K \in \D'(G;\End(V;E)).
\end{equation*}
This proves \eqref{conclusion2}. 
Now by Corollary~\ref{cor:kernel0}, we may write 
\begin{equation*}
\K = \K_{\text{homo}} + P(x) \log \|x\|,
\end{equation*} 
where \(\K_{\text{homo}} \in \D'(G;\End(V;E))\) is homogeneous of degree \(k-Q\), and \(P(x)\) is a homogeneous \(\End(V;E)\)--valued polynomial of degree \(k-Q\) if \(k \geq Q\), and zero otherwise.  
If \(\ell > k-Q\), then the formula of \(\K_{\gamma}\) shows that \(\K_{\gamma}\) is a homogeneous \(\End(V;E)\)--valued distribution of degree \(k-Q-\ell\). This is because at least one of the \(\ell\) derivatives in \(X^R_{\gamma}\) must hit \(\log \|x\|\) for it to give a non-zero contribution, which transforms this factor into something homogeneous. 
Finally, \(\K_{\gamma}\) agrees with an \(\End(V;E)\)--valued \(C^{\infty}\) function on \(G \setminus \{0\}\), because \(\K\) is \(C^{\infty}\) away from \(0\) by Corollary~\ref{cor:kernel0}. Thus if further \(\ell < k\), then \(\K_{\gamma}\) is a kernel of type \(k-\ell\). 
The last assertion in \eqref{cor:i} is now established.

In order to prove \eqref{cor:iii}, let \(\K^{\circ} \in \D'(G;\End(V;V))\) be the distribution in the proof of Corollary~\ref{cor:kernel0} so that 
\begin{equation} \label{eq:K0}
A^t(D) A(D) \K^{\circ} = \delta \otimes I \quad \text{on \(G\)}.
\end{equation}
This \(\K^{\circ}\) was constructed using Theorem~\ref{thm:kernel_prelim}, with \(A^t(D)A(D)\) in place of \(A(D)\) and \(2k\) in place of \(k\). Hence \(\K^{\circ}\) is \(C^{\infty}\) on \(G\setminus \{0\}\), and satisfies a pointwise bound
\begin{equation} \label{eq:K0bound}
\K^{\circ}(x) = 
\begin{cases}
O(1) &\quad \text{if \(2k < Q\)},\\
O\bigl(\|x\|^{2k-Q}(1+\log\|x\|)\bigr) &\quad \text{if \(2k \geq Q\)}.
\end{cases}
\end{equation} 
as \(\lVert x \rVert \to \infty\). Furthermore, we may write 
\begin{equation} \label{eq:K0expand}
\K^{\circ} = \K^{\circ}_{\text{homo}} + P^{\circ}(x) \log \|x\|,
\end{equation} 
where \(\K^{\circ}_{\text{homo}} \in \D'(G;\End(V;V))\) is homogeneous of degree \(2k-Q\), and \(P^{\circ}(x)\) is a homogeneous \(\End(V;V)\)--valued polynomial of degree \(2k-Q\) if \(2k \geq Q\), and zero otherwise.  

Now let \(A(D) = \sum_{\gamma \in \Ind_k} A^{\gamma} X_{\gamma}\) and let \(\tilde{\K} \in \D'(G;\End(E;V))\) be defined by
\begin{equation*}
\tilde{\K} \defeq (-1)^k \sum_{\gamma \in \Ind_k} X^R_{\gamma} \K^{\circ} (A^{\gamma})^t.
\end{equation*}
Suppose \(\psi \in D(G;V)\).
Then
\begin{equation*}
(A(D)\psi)*\tilde{\K} = (A^t(D)A(D)\psi)*\K^{\circ}.
\end{equation*}
As a result, if 
\begin{equation*}
w \defeq \psi - (A(D)\psi)*\tilde{\K} = \psi - (A^t(D) A(D)\psi)*\K^{\circ},
\end{equation*}
then from \eqref{eq:K0} we have
\begin{equation*}
A^t(D) A(D) w = 0.
\end{equation*}
In particular, hypoellipticity of \(A^t(D) A(D)\) implies that \(w \in C^{\infty}(G;V)\). The pointwise estimate for \(\K^{\circ}\) in \eqref{eq:K0bound} shows that, as \(\lVert x \rVert \to \infty\), 
\begin{equation*}
w(x) = 
\begin{cases}
O(1) &\quad \text{if \(2k < Q\)},\\
O\bigl(\|x\|^{2k-Q}(1+\log\|x\|)\bigr) &\quad \text{if \(2k \geq Q\)}.
\end{cases}
\end{equation*} 
In particular, \(w\) is an \(V\)--valued tempered distribution on \(G\). Using the Liouville--type theorem in Baldi, Franchi and Tesi \cite{BaldiFranchiTesi2009}*{Proposition 3.2}, we then see that \(w\) is a \(V\)--valued polynomial on \(G\). 

To proceed further, if \(\ell \in \N_0\), \(\ell > k - Q\) and \(\gamma \in \Ind_{\ell}\), we claim that \(X_{\gamma} w = 0\) on \(G\). Indeed, then
\begin{equation*}
X_{\gamma} w = X_{\gamma} \psi - (A(D)\psi)*\tilde{\K}_{\gamma}
\end{equation*}
where 
\begin{equation*}
\tilde{\K}_{\gamma} \defeq X_{\gamma} \tilde{\K} = (-1)^k \sum_{\gamma' \in \Ind_k} X_{\gamma} X^R_{\gamma'} \K^{\circ} (A^{\gamma'})^t \in \D'(G;\End(E;V)).
\end{equation*}
By \eqref{eq:K0expand}, we then have
\begin{equation} \label{eq:tildeKgammadef}
\tilde{\K}_{\gamma} = (-1)^k \sum_{\gamma' \in \Ind_k} X_{\gamma} X^R_{\gamma'}  \Big( \K^{\circ}_{\text{homo}} + P^{\circ}(x) \log \|x\| \Big) (A^{\gamma'})^t.
\end{equation}
However, since \(\ell > k - Q\), this shows \(\tilde{\K}_{\gamma}\) is homogeneous of degree 
\begin{equation*}
2k - Q - (\ell + k) = k - \ell - Q < 0.
\end{equation*} 
This is because when \(\ell > k - Q\), we have \(\ell + k > 2k - Q\), so at least one of the \(\ell + k\) derivatives in \(X_{\gamma} X^R_{\gamma'}\) must hit the factor \(\log \|x\|\) in \eqref{eq:tildeKgammadef} for it to give a non-zero contribution, making the term homogeneous. Furthermore, \(\tilde{\K}_{\gamma}\) agrees with an \(\End(E;V)\)--valued \(C^{\infty}\) function on \(G \setminus \{0\}\). Since \(\psi \in \D(G;V)\), we see that \((A(D)\psi)*\tilde{\K}_{\gamma}(x) \to 0\) as \(x \to \infty\). Thus \(X_{\gamma} w(x) \to 0\) as \(\|x\| \to \infty\) as well, and from the fact that \(w\) is a polynomial, our claim \(X_{\gamma} w = 0\) follows.

If now \(k < Q\), we apply the above claim with \(\ell = 0\), and see that \(w = 0\), i.e.\ 
\begin{equation*} 
  \psi - (A(D)\psi)*\tilde{\K} = 0.
\end{equation*} 
If \(k \geq Q\), we apply the above claim with \(\ell = k-Q+1\), and see that \(w = \psi - (A(D)\psi)*\tilde{\K}\) is a \(V\)--valued polynomial whose non-isotropic degree is \(\leq k-Q\). This establishes 
\eqref{eq_quohohw4she8Aphaedie9aja}.

If \(\ell > k - Q\) and \(\gamma \in \Ind_{\ell}\), applying \(X_{\gamma}\) to both sides of \eqref{eq_quohohw4she8Aphaedie9aja}, we see that \(X_{\gamma} \psi = (A(D)\psi)*\tilde{\K}_{\gamma}\), where as in the above, \(\tilde{\K}_{\gamma} \defeq X_{\gamma} \tilde{\K}\). Indeed all desired properties of \(\tilde{\K}_{\gamma}\) have already been established above. This establishes \eqref{conclusion4}.

Finally, if \(\gamma = (\gamma_1, \dots, \gamma_k) \in \Ind_k\), then \(\tilde{\K}_{\gamma} = X_{\gamma_1} \tilde{\K}_{(\gamma_2,\dots,\gamma_k)}\) is a left-invariant derivative of a kernel of type \(1\). Together with \eqref{conclusion4}, this implies \eqref{conclusion5}, as we have seen towards the end of the last section.
\end{proof}

We finish this section with some examples of maximally hypoelliptic operators on \(G\).

\begin{example}[Subelliptic gradient] \label{eg1}
Let \(V = \R\), \(E = \g_1^*\) and \(e^1, \dotsc, e^m\) be the basis of \(E\) dual to \(X_1, \dotsc, X_m\). 
Let \(A \in \End(V;E) \otimes T_1(\g_1)\) be such that 
\begin{equation}
A(D) u = \sum_{1 \leq j \leq m} X_j u \, e^j
\end{equation} 
if \(u \in \D'(G;V)\). 
Then \(A(D) \colon \D'(G;\R) \to \D'(G;\g_1^*)\) is obviously maximally hypoelliptic.
\end{example}

\begin{example} \label{eg2}
More generally, let \(k \in \N\), \(V = \R\), \(E = \g_1^*\) and \(A \in \End(V;E) \otimes T_k(\g_1)\) be  such that 
\begin{equation} \label{high_order_gradient}
  A(D) u = \sum_{1 \leq j \leq m} X_j^k u \, e^j
\end{equation} 
if \(u \in \D'(G;V)\). 
(This generalizes the previous example, which is the case \(k = 1\).) We will show that \(A(D) \colon \D'(G;\R) \to \D'(G;\g_1^*)\) is maximally hypoelliptic.

To see this, note that \(A^t(D) A(D) \colon \D'(G;\R) \to \D'(G;\R)\) is given by 
\begin{equation*}
A^t(D) A(D) u = (X_1^{2k} + \dots + X_m^{2k}) u, \quad u \in \D'(G;\R)
\end{equation*}
which is hypoelliptic by a theorem of Helffer and Nourrigat \cite{MR537467} (see also Melin \cite{MR739894}). Indeed, as in Folland and Stein \cite{MR657581}*{(4.20)}, suppose \(\pi\) is an irreducible (complex) representation of \(G\), which determines a representation \(d\pi\) of the Lie algebra \(\g\) as skew-Hermitian operators on \(\mathcal{S}_{\pi}\). Then \(d\pi\) extends to a representation of \(T_k (\g)\) (still denoted \(d\pi\)) as operators acting on \(\mathcal{S}_{\pi}\), and if \(d\pi(X_1^{2k} + \dots + X_m^{2k})v = 0\) for some \(v \in \mathcal{S}_{\pi}\), then 
\begin{equation*}
0 = (d\pi(X_1^{2k} + \dots + X_m^{2k})v, v ) = \sum_{i=1}^m (d\pi(X_i)^k v, d\pi(X_i)^k v)
\end{equation*}
so \(d\pi(X_i)^k v = 0\) for \(1 \leq i \leq m\). It follows that whenever \(\ell \in \N_0\) and \(2^{\ell} \geq k\), then \(d\pi(X_i)^{2^{\ell}}v = 0\) for \(1 \leq i \leq m\). A similar argument as above then shows that \(d\pi(X_i) v = 0\) for \(1 \leq i \leq m\). Since \(X_1,\dots,X_m\) generates \(\g\), this shows that \(d\pi(X)v = 0\) for all \(X \in \g\), so either \(\pi\) is the trivial representation, or \(v = 0\). This verifies what is called the Rockland condition, and the aforementioned theorem of Helffer and Nourrigat implies that \(X_1^{2k} + \dots + X_m^{2k}\) is hypoelliptic. It follows now from our Theorem~\ref{thm:kernel_main} that \(A(D)\) given by \eqref{high_order_gradient} is also maximally hypoelliptic.
\end{example}

\begin{example}[Korn--Sobolev] \label{eg3}
Let \(V = \g_1\), \(E = S_2(\g_1)\) the subspace of all symmetric tensors in \(T_2(\g_1)\) and \(e^1, \dotsc, e^m\) be a basis of \(V\). 
Write \(e^{ij}\) as a shorthand for \(\operatorname{Sym}(e^i \otimes e^j)\), so that \(\{e^{ij}\}_{1 \leq i \leq j \leq m}\) is a basis for \(E\). 
Let \(A \in \End(V;E) \otimes T_1(\g_1)\) be the Korn operator
\begin{equation*}
A(D) u = \sum_{1 \leq i \leq j \leq m} (X_i u_j + X_j u_i) e^{ij}
\end{equation*}
if \(u = \sum_{1 \leq j \leq m} u_j e^j \in \D'(G;V)\). We will show that \(A(D) \colon \D'(G;\R) \to \D'(G;S_2(\g_1))\) is maximally hypoelliptic. 

Fix \(1 \leq i < j \leq m\), and let \(u \in \D(G;V)\). We want to estimate \(\|X_i u_j\|_{L^2(G)}\) in terms of \(\|A(D) u\|_{L^2(G;E)}\). 
But pick \(k = 2r\) where \(r\) is the step of the group \(G\) (so that \(\g_r \ne \{0\}\) in \eqref{eq:Liealg_grading}); we assume \(r \ge 2\) for otherwise we are in the Euclidean situation and the ellipticity of \(A(D)\) is well-known. We will express \((X_{\ell}^k X_i u_j)_{1 \leq {\ell} \leq m}\) in terms of order \(k\) derivatives of \(A(D) u\). A tool that comes in handy is the following observation: for \(1 \leq {\ell}, {\ell'} \leq m\) and \(s \in \N\),
\begin{equation*}
X_{\ell}^{\otimes s} \otimes X_{\ell'} \in \g_{s+1} + \g_s \otimes X_{\ell} + \dots + \g_1 \otimes X_{\ell}^{\otimes s}
\end{equation*}
which can be proved easily using induction on \(s\); indeed,
\begin{equation*}
\begin{split}
X_{\ell}^s X_{\ell'} = & \binom{s}{0} (\ad X_{\ell})^s (X_{\ell'}) + \binom{s}{1} (\ad X_{\ell})^{s-1} (X_{\ell'}) X_{\ell} + \dots \\
& \qquad + \binom{s}{s-1} (\ad X_{\ell})(X_{\ell'}) X_{\ell}^{s-1} + \binom{s}{s} X_{\ell'} X_{\ell}^s
\end{split}
\end{equation*}
which is a general identity that holds in all associative algebras.
Since \(\g_{r+1} = \{0\}\), this shows
\begin{equation} \label{eq:commutator}
X_{\ell}^r X_{\ell'} = C_{{\ell},\ell'}(D) X_{\ell} 
\quad \text{where} \quad C_{{\ell},\ell'} \in \g_r + \g_{r-1} \otimes X_{\ell} + \dots + \g_1 \otimes X_{\ell}^{\otimes (r-1)}.
\end{equation}

Recall we fixed \(1 \leq i < j \leq m\), and wrote \(k = 2r\). Let now \(1 \leq {\ell} \leq m\). We will express \(X_{\ell}^k (X_i u_j)\) in terms of order \(k\) derivatives of \(A(D)u\).
\begin{description}
\item[Case 1] \({\ell} = j\) 

Then 
\begin{equation*}
X_{\ell}^r (X_i u_j) = X_j^r (X_i u_j) = C_{j,i}(D) (X_j u_j)
\end{equation*}
by \eqref{eq:commutator}, so 
\begin{equation} \label{case1}
X_{\ell}^k (X_i u_j) = X_{\ell}^r C_{j,i}(D) (X_j u_j).
\end{equation} 

\item[Case 2] \({\ell} = i\) 

Then
\begin{equation*}
\begin{split}
X_{\ell}^r (X_i u_j) = X_i^r (X_i u_j) &= X_i^r (X_i u_j + X_j u_i) - X_i^r X_j u_i \\
&= X_i^r (X_i u_j + X_j u_i) - C_{i,j}(D) (X_i u_i) 
\end{split}
\end{equation*}
by \eqref{eq:commutator}, so
\begin{equation} \label{case2}
X_{\ell}^k (X_i u_j) = X_i^{2r} (X_i u_j + X_j u_i) - X_i^r C_{i,j}(D) (X_i u_i).
\end{equation} 

\item[Case 3] \(1 \le {\ell} \le m\), \({\ell} \ne i\) nor \(j\) 

Then by \eqref{eq:commutator},
\begin{equation*}
X_{\ell}^r (X_i u_j) = C_{{\ell},i}(D) (X_{\ell} u_j) =  C_{{\ell},i}(D) (X_{\ell} u_j + X_j u_{\ell}) -  C_{{\ell},i}(D) X_j u_{\ell}.
\end{equation*}
But 
\begin{equation*}
X_{\ell}^r C_{{\ell},i}(D) = \tilde{C}_{{\ell},i}(D) X_{\ell}^r \quad \text{for some} \quad \tilde{C}_{{\ell},i} \in  \g_r + \g_{r-1} \otimes X_{\ell} + \dots + \g_1 \otimes X_{\ell}^{\otimes (r-1)}.
\end{equation*}
As a result,
\begin{equation*}
X_{\ell}^k (X_i u_j) 
= X_{\ell}^r C_{{\ell},i}(D) (X_{\ell} u_j + X_j u_{\ell}) - \tilde{C}_{{\ell},i}(D) X_{\ell}^r X_j u_{\ell}
\end{equation*}
which in light of \eqref{eq:commutator} gives
\begin{equation} \label{case3}
X_{\ell}^k (X_i u_j) = X_{\ell}^r C_{{\ell},i}(D) (X_{\ell} u_j + X_j u_{\ell}) - \tilde{C}_{{\ell},i}(D) C_{{\ell},j}(D) (X_{\ell} u_{\ell}).
\end{equation} 
\end{description}

Now since \(\dim \g_1 \ge 2\) and \(\dim \g_j \ge 1\) for \(2 \leq j \leq r\), we have
\begin{equation*}
Q = \sum_{j=1}^r j \cdot \dim \g_j \geq 2 + \sum_{j=2}^r j = \frac{r^2+r+2}{2} \geq 2r,
\end{equation*}
with a strict inequality unless both \(r = 2\) and \(\dim \g_1 = 2\). If \(Q > 2r = k\), then using the ellipticity of \(u \mapsto (X_1^k u, \dots, X_m^k u)\) as in Example~\ref{eg2}, and using \eqref{eq_quohohw4she8Aphaedie9aja} in Theorem~\ref{thm:kernel_main}, we conclude the existence of \(\tilde{\K}_1, \dots, \tilde{\K}_m \in \D'(G;\R)\), each a kernel of type \(k\), such that 
\[
X_i u_j = \sum_{{\ell}=1}^m (X_{\ell}^k X_i u_j)*\tilde{\K}_{\ell}.
\]
But using \eqref{case1}, \eqref{case2} and \eqref{case3}, we have expressed \(X_{\ell}^k X_i u_j\) as a derivative of order \(k\) of components of \(A(D)u\). As a result, we may express \(X_i u_j\) as the convolution of \(A(D)u\) with the order \(k\) derivative of a kernel of type \(k\). As observed at the end of the last section, this shows
\begin{equation} \label{eq:Korn_Sob_ellipticity}
\sum_{i=1}^{\dim V} \|X_i u\|_{L^2(G;V)} \leq \|A(D)u\|_{L^2(G;E)}.
\end{equation}
On the other hand, if indeed \(r = 2\) and \(\dim \g_1 = 2\), then \(Q = 4\) and we didn't have Case 3 above. So we may express \( (X_1^2 (X_i u_j), X_2^2 (X_i u_j)) \) in terms of second order derivatives of \(A(D)u\), which then allows us to express \(X_i u_j\) as the convolution of \(A(D)u\) with the second order derivative of a kernel of type \(2\). Hence \eqref{eq:Korn_Sob_ellipticity} also holds in this case, and this completes our proof of the maximal hypoellipticity of \(A(D)\).
\end{example}

\section{\texorpdfstring{\(L^1\)}{L¹} estimates}

In this section we prove \(L^1\) estimates for certain \(A(D)\) satisfying the conditions of Theorem~\ref{thm:kernel_main}.

To set up some notations, let \(S_{\ell}(\g_1)\) be the subspace of all symmetric tensors in \(T_{\ell}(\g_1)\). There is a symmetrization map \(\operatorname{Sym} \colon T_{\ell}(\g_1) \to S_{\ell}(\g_1)\), which is a linear surjection given by
\begin{equation*}
X^{\otimes}_{\lambda} \mapsto \frac{1}{\ell!} \sum_{\sigma \in S_{\ell}} X^{\otimes}_{\sigma(\lambda)}, \quad \lambda \in \Ind_{\ell},
\end{equation*}
where \(S_{\ell}\) is the symmetric group on \(\ell\) elements, and \(\sigma(\lambda) \defeq (\lambda_{\sigma(1)}, \dotsc, \lambda_{\sigma(\ell)})\) if \(\lambda = (\lambda_1, \dotsc, \lambda_{\ell}) \in \Ind_{\ell}\) and \(\sigma \in S_{\ell}\). Let \(I_{\ell} \defeq \{\beta = (\beta_1, \dotsc, \beta_m) \in \N_0^m \colon |\beta| = \ell\}\) for the set of all multiindices of length \(\ell\). 
For \(\lambda \in \Ind_{\ell}\) and \(\beta \in I_{\ell}\), we will write \(\beta = \operatorname{Sym}(\lambda)\) if for \(1 \leq j \leq m\), \(\beta_j\) is the number of indices among \(\lambda_1, \dotsc, \lambda_{\ell}\) that is equal to~\(j\). 
If \(\beta \in I_{\ell}\), it will also be convenient to write
\begin{equation*}
\tilde{X}^{\otimes \beta} = \operatorname{Sym} (X^{\otimes}_{\lambda})
\end{equation*}
where \(\lambda\) is any element in \(\Ind_{\ell}\) with \(\operatorname{Sym}(\lambda) = \beta\). Note that \(S_{\ell}(\g_1)\) is isomorphic to the vector space of all  commutative homogeneous polynomials of \(\xi = (\xi_1, \dotsc, \xi_m) \in \R^m\) of degree \(\ell\) with real coefficients, via the map
\begin{equation*}
\tilde{X}^{\otimes \beta} \mapsto \xi^{\beta}, \quad \beta \in I_{\ell}
\end{equation*}
where \(\xi^{\beta} \defeq \xi_1^{\beta_1} \dots \xi_m^{\beta_m}\) for \(\beta \in I_{\ell}\). 

Let now \(F\) be a finite dimensional inner product space over \(\R\). The symmetrizaiton map \(\operatorname{Sym} \colon T_{\ell}(\g_1) \to S_{\ell}(\g_1)\) extends to a linear map \begin{equation*}\operatorname{Sym} \colon \End(E;F) \otimes T_{\ell}(\g_1) \to \End(E;F) \otimes S_{\ell}(\g_1);\end{equation*} more explicitly, for \(L \in \End(E;F) \otimes T_{\ell}(\g_1)\), if \(L = \sum_{\lambda \in \Ind_{\ell}} B^{\lambda} X^{\otimes}_{\lambda}\) where each \(B^{\lambda} \in \End(E;F)\), we have
\begin{equation*}
\operatorname{Sym}(L) =  \sum_{\lambda \in \Ind_{\ell}} B^{\lambda} \operatorname{Sym}(X^{\otimes}_{\lambda}) = \sum_{\beta \in I_{\ell}} \tilde{B}_{\beta} \tilde{X}^{\otimes \beta}
\end{equation*}
where
\begin{equation*}
\tilde{B}_{\beta} \defeq \sum_{\lambda \in \Ind_{\ell} \colon \operatorname{Sym}(\lambda) = \beta} B^{\lambda} \quad \text{for each \(\beta \in I_{\ell}\)}.
\end{equation*}
As before, we may associate to each such \(\operatorname{Sym}(L)\) a homogeneous left-invariant linear partial differential operator \(\operatorname{Sym}(L)(D)\) of order \(\ell\) on \(G\). Moreover, by identifying \(S_{\ell}(\g_1)\) with the space of all commutative homogeneous degree \(\ell\) polynomials on \(\R^m\), we may identify \(\operatorname{Sym}(L) \in \End(E;F) \otimes S_{\ell}(\g_1)\) with
\begin{equation} \label{eq:SymL_symbolexp}
\operatorname{Sym}(L)(\xi) = \sum_{\beta \in I_{\ell}} \tilde{B}_{\beta} \xi^{\beta}.
\end{equation}
This polynomial in \(\xi\) is usually called the \emph{symbol} of \(\operatorname{Sym}(L)(D)\); we say \(\operatorname{Sym}(L)(D)\) is cocanceling, if and only if
\begin{equation*}
\bigcap_{\xi \in \R^m} \text{ker} \, \operatorname{Sym}(L)(\xi) = \{0\}.
\end{equation*}
This is the same as saying that 
\begin{equation*}
\bigcap_{\beta \in I_{\ell}} \text{ker} \, \tilde{B}_{\beta} = \{0\},
\end{equation*}
if \(\operatorname{Sym}(L)(\xi)\) is as in (\ref{eq:SymL_symbolexp}); this follows, for instance, from a variant of the proof of Lemma~\ref{lem:cocan_open} below.

One of our main results is as follows.

\begin{theorem} \label{thm:can_est}
Let \(k \in \N\), let \(V, E\) be finite dimensional inner product spaces over \(\R\), and let \(A \in \End(V;E) \otimes T_k(\g_1)\). Suppose \(A(D)\) is maximally hypoelliptic.
Assume there exist \(L \in \End(E;F) \otimes T_{\ell}(\g_1)\) for some finite dimensional real inner product space \(F\) and some \(\ell \in \N\), such that 
\begin{equation*}L(D) \circ A(D) = 0,\end{equation*}
and that \(\operatorname{Sym}(L)(D)\) is cocanceling.
Then for any \(\gamma \in \Ind_{k-1}\), we have
\begin{equation*}
\|X_{\gamma} u\|_{L^{\frac{Q}{Q-1}}(G;V)} \leq C \| A(D) u \|_{L^1(G;E)}
\end{equation*}
for all \(u \in C^{\infty}_c(G;V)\).
\end{theorem}

The proof of Theorem~\ref{thm:can_est} depends on Theorem~\ref{thm:kernel_main}(\ref{cor:i}), as well as the following proposition, in the spirit of \cite{VanSchaftingen_2013}:
\begin{proposition} \label{prop:cocan_est}
Suppose \(L \in \End(E;F) \otimes T_{\ell}(\g_1)\), such that \(\operatorname{Sym}(L)(D)\) is cocanceling. Suppose \(f \in C^{\infty}(G;E)\) is such that \(L(D)f = 0\). Then for any \(\phi \in C^{\infty}_c(G;E)\),
\begin{equation*}
\left| \int_G \langle f, \phi \rangle_E \right| \lesssim \|f\|_{L^1(G;E)} \|D \phi\|_{L^Q(G;E)}.
\end{equation*}
\end{proposition}

To prove Proposition~\ref{prop:cocan_est}, we need Theorem 5.3 of \cite{Chanillo_VanSchaftingen_2009} (see also \cites{VanSchaftingen_2004_L1,VanSchaftingen_2008} for its Euclidean precedent). We reformulate it as follows:

\begin{lemma} \label{lem:CVS}
Suppose \(f \in C^{\infty}(G; T_{\ell}(\g_1))\) is given by 
\begin{equation*}
f = \sum_{\lambda \in \Ind_{\ell}} f^{\lambda} X^{\otimes}_{\lambda}
\end{equation*}
where each \(f^{\lambda} \in C^{\infty}(G; \R)\). Assume that
\begin{equation*}
\sum_{\lambda \in \Ind_{\ell}} X_{\lambda} f^{\lambda} = 0.
\end{equation*}
Then for any \(\varphi \in C^{\infty}_c(G; S_{\ell}(\g_1))\), we have
\begin{equation*}
\left| \int_G \langle f, \varphi \rangle_{T_{\ell}(\g_1)} \right| \lesssim \|f\|_{L^1(G;T_{\ell}(\g_1))} \|D \varphi \|_{L^Q(G;S_{\ell}(\g_1))}
\end{equation*}
where \(\langle \cdot, \cdot \rangle_{T_{\ell}(\g_1)}\) is an inner product on the inner product space \(T_{\ell}(\g_1)\).
\end{lemma}

The above lemma easily generalizes to the situation where \(f\) and \(\varphi\) takes value in \(F \otimes T_{\ell}(\g_1)\) and \(F \otimes S_{\ell}(\g_1)\) for some finite dimensional inner product space \(F\) over \(\R\). More precisely, we will need the following corollary of Lemma~\ref{lem:CVS}:

\begin{corollary} \label{cor:cocan_est}
Let \(F\) be a finite dimensional inner product space over \(\R\). Suppose \(f = \sum_{\lambda \in \Ind_{\ell}} f^{\lambda} X^{\otimes}_{\lambda} \in C^{\infty}(G; F \otimes T_{\ell}(\g_1))\) is such that
\begin{equation*}
\sum_{\lambda \in \Ind_{\ell}} X_{\lambda} f^{\lambda} = 0
\end{equation*}
componentwise. Then for any \(\beta \in I_{\ell}\) and any \(\phi \in C^{\infty}_c(G;F)\), we have
\begin{equation*}
\left| \int_G \langle \tilde{f}_{\beta}, \phi \rangle_F \right| \lesssim \sum_{\beta \in I_{\ell}} \|\tilde{f}_{\beta}\|_{L^1(G;F)} \|D \phi \|_{L^Q(G;F)}
\end{equation*}
where \(\langle \cdot, \cdot \rangle_F\) is an inner product on the inner product space \(F\). Here \begin{equation*}\tilde{f}_{\beta} \defeq \sum_{\lambda \in \Ind_{\ell} \colon \operatorname{Sym}(\lambda) = \beta} f^{\lambda}.\end{equation*}
\end{corollary}

\begin{proof}
It suffices to take \(\varphi = \sum_{\lambda \in \Ind_{\ell} \colon \operatorname{Sym}(\lambda) = \beta} \phi  X^{\otimes}_{\lambda} \in C^{\infty}(G; S_{\ell}(\g_1))\) and apply Lemma~\ref{lem:CVS} to each \(F\)-component.
\end{proof}

\begin{proof}[Proof of Proposition~\ref{prop:cocan_est}]
Let \(L = \sum_{\lambda \in \Ind_{\ell}} B^{\lambda} X^{\otimes}_{\lambda} \in \End(E;F) \otimes T_{\ell}(\g_1)\), and \(\operatorname{Sym}(L)(\xi) = \sum_{\beta \in I_{\ell}} \tilde{B}_{\beta} \xi^{\beta}\). Since \(\operatorname{Sym}(L)(D)\) is cocanceling, the map 
\begin{equation*}
e \mapsto (\tilde{B}_{\beta}(e))_{\beta \in I_{\ell}}
\end{equation*}
is an injective linear map from \(E\) to \(F^N\) where \(N\) is the number of elements in \(I_{\ell}\). This map has a left inverse, i.e.\ for every \(\beta \in I_{\ell}\), there exists a linear map \(C_{\beta} \colon F \to E\) such that
\begin{equation*}
e = \sum_{\beta \in I_{\ell}} C_{\beta} \tilde{B}_{\beta} e
\end{equation*}
for all \(e \in E\). It follows that for any \(f \in C^{\infty}(G;E)\), we have 
\begin{equation*}
f(x) = \sum_{\beta \in I_{\ell}} C_{\beta} \tilde{B}_{\beta} f(x)
\end{equation*} 
for all \(x \in G\). Now
\begin{equation*}
\langle f, \phi \rangle_E = \sum_{\beta \in I_{\ell}} \langle \tilde{B}_{\beta} f, C_{\beta}^* \phi \rangle_F,
\end{equation*}
where \(C_{\beta}^* \colon E \to F\) is the adjoint to \(C_{\beta}\). The condition \(L(D)f = 0\) guarantees that 
\begin{equation*}
\sum_{\lambda \in \Ind_{\ell}} X_{\lambda} B^{\lambda} f = 0.
\end{equation*}
Thus Corollary~\ref{cor:cocan_est} applies, and we obtain the bound
\begin{equation*}
\left| \int_G \langle \tilde{B}_{\beta} f, C_{\beta}^* \phi \rangle_F \right| \lesssim \|f\|_{L^1(G;E)} \|D \phi\|_{L^Q(G;E)}
\end{equation*}
for all \(\beta \in I_{\ell}\). Summing over \(\beta\) gives the desired estimate for \(\int_G \langle f, \phi \rangle_E\).
\end{proof}

We may now prove Theorem~\ref{thm:can_est}. 
\begin{proof}[Proof of Theorem~\ref{thm:can_est}]
By Theorem~\ref{thm:kernel_main}(\ref{cor:i}), for any \(\gamma \in \Ind_{k-1}\), and any \(\psi \in C^{\infty}_c(G;V)\), we have
\begin{equation*}
X_{\gamma}^t \psi = A^t(D) [\psi * \K_{\gamma}]
\end{equation*}
where \(\K_{\gamma}\) is a \(\End(V;E)\)--valued kernel of type \(1\). Hence for any \(u \in C^{\infty}_c(G;V)\), we have
\begin{equation*}
\langle X_{\gamma} u, \psi \rangle_{L^2(G;V)} 
= \langle A(D) u, \psi * \K_{\gamma} \rangle_{L^2(G;E)}
\end{equation*}
We may then apply Proposition~\ref{prop:cocan_est} to \(f\defeq A(D)u\) and \(\phi\defeq \psi * \K_{\gamma} \). Since there exists \(L \in \End(E;F) \otimes S_{\ell}(\g_1)\), such that \(L(D) \circ A(D) = 0,\) and that \(\operatorname{Sym}(L)(D)\) is cocanceling, the above gives
\begin{align*}
\left| \langle X_{\gamma} u, \psi \rangle_{L^2(G;V)} \right| 
&\leq C \|A(D) u\|_{L^1(G;E)} \|D (\psi * \K_{\gamma} )\|_{L^Q(G;V)} \\
&\leq C \|A(D) u\|_{L^1(G;E)} \|\psi\|_{L^Q(G;V)}.\qedhere
\end{align*}
\end{proof}

\section{Hardy inequalities}

In this section we prove certain Hardy inequalities under the same assumptions as in Theorem~\ref{thm:can_est}, in the spirit of \cite{MR3283556}. 

\begin{theorem} \label{thm:Hardy}
Let \(k \in \N\), let \(V, E\) be finite dimensional inner product spaces over \(\R\), and let \(A \in \End(V;E) \otimes T_k(\g_1)\). Suppose \(A(D)\) is maximally hypoelliptic. 
Assume there exists \(L \in \End(E;F) \otimes T_{\ell}(\g_1)\) for some finite dimensional real inner product space \(F\) and some \(\ell \in \N\), such that 
\begin{equation*}L(D) \circ A(D) = 0,\end{equation*}
and that \(\operatorname{Sym}(L)(D)\) is cocanceling.
Then 
\begin{enumerate}[(a)]
\item \label{it_ainie9yoh5cau5Ieniid6aiv} for any \(\ell \in \{1, \dotsc, \min\{k,Q-1\} \}\), and any \(p \in [1, \frac{Q}{Q-\ell} )\), we have
\begin{equation*}
\left( \int_G \left(\|x\|^{Q-\ell} |D^{k-\ell} u(x)|_V \right)^p \frac{dx}{\|x\|^Q} \right)^{1/p} \leq C \| A(D) u \|_{L^1(G;E)}
\end{equation*}
for all \(u \in C^{\infty}_c(G;V)\); furthermore,
\item \label{it_Thah7ahshae4ozeujoh5thae} if \(k \geq Q\), then 
\begin{equation*}
\|D^{k-Q} u\|_{L^{\infty}(G;V)} \leq C \| A(D) u \|_{L^1(G;E)} 
\end{equation*}
for all \(u \in C^{\infty}_c(G;V)\).
\end{enumerate}
\end{theorem}

The proof relies on Theorem~\ref{thm:kernel_main}(\ref{cor:iii}), as well as the following proposition:

\begin{proposition} \label{prop:Hardy_dual}
Let \(E, F\) be finite dimensional inner product spaces over \(\R\), and \(\ell \in \N\). Suppose \(L \in \End(E;F) \otimes T_{\ell}(\g_1)\), such that \(\operatorname{Sym}(L)(D)\) is cocanceling. Suppose \(f \in C^{\infty}(G;E)\) is such that \(L(D)f = 0\). Then for any \(\phi \in C^{\infty}_c(G, E)\), we have
\begin{equation*}
\left | \int_G \langle f, \phi \rangle_E \right |  \leq C \sum_{j=1}^{\ell} \int_G |f(x)|_E |D^j \phi(x)|_E \|x\|^j dx.
\end{equation*}
\end{proposition}

The key here is that the sum on the right hand side begins with \(j = 1\) (instead of \(j = 0\)).

\begin{proof}[Proof of Proposition~\ref{prop:Hardy_dual}]
Let \(L = \sum_{\lambda \in \Ind_{\ell}} B^{\lambda} X^{\otimes}_{\lambda} \in \End(E;F) \otimes T_{\ell}(\g_1)\), and \(\operatorname{Sym}(L)(\xi) = \sum_{\beta \in I_{\ell}} \tilde{B}_{\beta} \xi^{\beta}\). Since \(\operatorname{Sym}(L)(D)\) is cocanceling, we may construct, as in the proof of Proposition~\ref{prop:cocan_est}, a linear map \(C_{\beta} \colon F \to E\) for every \(\beta \in I_{\ell}\) such that 
\begin{equation*}
e = \sum_{\beta \in I_{\ell}} C_{\beta} \tilde{B}_{\beta} e
\end{equation*}
for all \(e \in E\). Now write \(x^{\beta}\) for \(x_1^{\beta_1} \dots x_m^{\beta_m}\) for \(\beta \in I_{\ell}\). Then
\begin{equation*}
\frac{1}{\beta!} L^t(D) (x^{\beta})  = \tilde{B}_{\beta}^* 
\end{equation*}
for all \(\beta \in I_{\ell}\), so
\begin{equation*}
\phi(x) = L^t(D) \sum_{\beta \in I_{\ell}} \frac{1}{\beta!} (x^{\beta} C_{\beta}^* \phi(x)) + O\left ( \sum_{j=1}^{\ell} \|x\|^j |D^j \phi(x)|_E \right )
\end{equation*}
for all \(x \in G\). Plugging this back into \(\int_G \langle f, \phi \rangle_E\), and noting that \(L(D) f = 0\), we have
\begin{equation*}
\left | \int_G \langle f, \phi \rangle_E \right | 
\leq C \sum_{j=1}^{\ell} \int_G |f(x)|_E |D^j \phi(x)|_E \|x\|^j dx,
\end{equation*}
as desired.
\end{proof}

\begin{proof}[Proof of Theorem~\ref{thm:Hardy}]
For \eqref{it_ainie9yoh5cau5Ieniid6aiv}, first let \(\ell \in \{1,\dotsc,\min\{k,Q-1\}\}\) so that \(k-Q < k - \ell < k\). By Theorem~\ref{thm:kernel_main}(\ref{cor:iii}), for any \(\gamma \in \Ind_{k-\ell}\), and any \(u \in C^{\infty}_c(G;V)\), we have
\begin{equation*}
X_{\gamma} u(x)
= \int_G \tilde{\K}_{\gamma}(y^{-1} x) [A(D) u](y) dy
\end{equation*}
for all \(x \in G\), where \(\tilde{\K}_{\gamma}\) is an \(\End(E;V)\)--valued function, that is \(C^{\infty}\) on \(G \setminus \{0\}\) and homogeneous of degree \(k-Q-(k-\ell)= -(Q-\ell)\). Let \(\rho \in C^{\infty}_c(\R;\R)\) be a cut-off function, so that \(\rho(t) = 1\) for \(|t| \leq 1/4\), and \(\rho(t) = 0\) for \(|t| \geq 1/2\). We then have \(X_{\gamma} u(x) = \mathbf{I}(x) + \mathbf{II}(x)\), where
\begin{equation*}
\mathbf{I}(x) \defeq \int_G \rho\left( \frac{\|y\|}{\|x\|} \right) \tilde{\K}_{\gamma}(x) [A(D) u](y) dy,
\end{equation*}
and
\begin{equation*}
\mathbf{II}(x) \defeq \int_G \left[ \tilde{\K}_{\gamma}(y^{-1} x) - \rho\left( \frac{\|y\|}{\|x\|} \right) \tilde{\K}_{\gamma}(x) \right] [A(D) u](y) dy.
\end{equation*}
To estimate \(\mathbf{I}(x)\), since \(L(D) \circ A(D) = 0\) and \(\operatorname{Sym}(L)(D)\) is cocanceling, we may then apply Proposition~\ref{prop:Hardy_dual} to \(f(y)\defeq [A(D)u](y)\) and \(\phi(y)\defeq \rho\left( \frac{\|y\|}{\|x\|} \right) \tilde{\K}_{\gamma}(x)\). Hence
\begin{equation*}
\|x\|^{Q-\ell} |\mathbf{I}(x)|_V \leq \int_{\|y\| \leq \frac{\|x\|}{2}} \frac{\|y\|}{\|x\|} |A(D) u(y)|_E dy.
\end{equation*} 
Since \(p \geq 1\), from Minkowski inequality, it follows that
\begin{equation} \label{eq:Hardy_I_est}
\left( \int_G \left(\|x\|^{Q-\ell} |\mathbf{I}(x)|_V \right)^p \frac{dx}{\|x\|^Q} \right)^{1/p} \leq C \| A(D) u \|_{L^1(G;E)}.
\end{equation}
Next,
\begin{align*}
\|x\|^{Q-\ell} |\mathbf{II}(x)|_V &\leq \int_{\|y\| \leq \frac{\|x\|}{2}} \frac{\|y\|}{\|x\|} |A(D) u(y)|_E dy \\
&\quad +  \int_{\|y\| \geq \frac{\|x\|}{2}} \frac{\|x\|^{Q-\ell}}{\|y^{-1} x\|^{Q-\ell}} |A(D) u(y)|_E dy.
\end{align*}
Since \(p \in [1, \frac{Q}{Q-\ell})\), from Minkowski inequality again, it follows that
\begin{equation} \label{eq:Hardy_II_est}
\left( \int_G \left(\|x\|^{Q-\ell} |\mathbf{II}(x)|_V \right)^p \frac{dx}{\|x\|^Q} \right)^{1/p} \leq C \| A(D) u \|_{L^1(G;E)}.
\end{equation}
Combining (\ref{eq:Hardy_I_est}) and (\ref{eq:Hardy_II_est}), we get the desired conclusion in part (a).

For \eqref{it_Thah7ahshae4ozeujoh5thae}, note that 
\begin{equation*}
D^{k-Q} u(0) = -\int_0^{\infty} \frac{d}{d \lambda} D^{k-Q} u(\delta_{\lambda} x_0) d\lambda
\end{equation*}
for any \(x_0 \in G\) with \(\|x_0\|=1\). Hence
\begin{equation*}
|D^{k-Q} u(0)|_V \leq \int_0^{\infty} \sum_{j=1}^r \lambda^{j-1} |D^j D^{k-Q} u(\delta_{\lambda} x_0)|_V d \lambda
\end{equation*}
for any such \(x_0\). Integrating over all such \(x_0\), we see that
\begin{equation*}
|D^{k-Q} u(0)|_V \leq C \int_G \sum_{j=1}^r \frac{ |D^{k-(Q-j)} u(x)|_V }{\|x\|^{Q-j}} dx.
\end{equation*}
Hence by part \eqref{it_ainie9yoh5cau5Ieniid6aiv} above (with \(p = 1\)), we have
\begin{equation*}
|D^{k-Q} u(0)|_V \lesssim \|A(D) u\|_{L^1(G;E)},
\end{equation*}
proving part \eqref{it_Thah7ahshae4ozeujoh5thae}.  
\end{proof}

\section{Construction of a compatible \texorpdfstring{\(L(D)\)}{L(D)}}
 
In applying Theorem~\ref{thm:can_est} and Theorem~\ref{thm:Hardy}, we need to find some \(L(D)\) such that \(L(D) \circ A(D) = 0\) and that \(\operatorname{Sym}(L)(D)\) is cocanceling. In that regard, we remark that the left-invariant differential operators on \(G\) form a left Noetherian ring \cite{Knapp2002}*{proposition 3.27 and problems 3.11--3.13}, and the left-invariant differential operators \(K (D) \colon C^{\infty}(G;E) \to C^{\infty}(G;\R)\) such that 
\begin{equation*}
  K (D) \circ A (D) = 0
\end{equation*}
form a left module over this left Noetherian ring. Hence by a non-commutative version of the Hilbert basis theorem, this module is finitely left-generated. Let us think of each \(K(D)\) in this module as a row vector, multiplying the matrix \(A(D)\) on the left, and let \(L (D)\) be a matrix of left-invariant differential operators so that the rows generate this module. By multiplying some further left-invariant differential operators, we can make \(L (D)\) a homogeneous operator while maintaining \(L(D) \circ A(D) = 0\). The question is then whether one can find such an \(L(D)\) such that \(\operatorname{Sym}(L)(D)\) is cocanceling. Below we develop a robust way that works for our examples of interest.

\begin{proposition} \label{prop:cocan_preserve}
Suppose \(L_0 \in \End(E;F) \otimes T_{\ell}(\g_1)\), and suppose \(\operatorname{Sym}(L_0)(D)\) is cocanceling. Let \(M \in \End(F;W) \otimes T_{\ell'}(\g_1)\) for some finite dimensional real inner product space \(W\) and some \(\ell' \in \N\). If there exists \(\xi_0 \in \R^m\) such that \(\operatorname{Sym}(M)(\xi_0)\) is injective, then \(\operatorname{Sym}(M \circ L_0)(D)\) is cocanceling.
\end{proposition}

Here \begin{equation*}(M \circ L_0) \defeq \sum_{\gamma' \in \Ind_{\ell'}} \sum_{\gamma \in \Ind_{\ell}} (M^{\gamma'} \circ L_0^{\gamma}) X^{\otimes}_{\gamma'} \otimes X^{\otimes}_{\gamma} \in \End(E;W) \otimes T_{\ell+\ell'}(\g_1),\end{equation*} if \(M = \sum_{\gamma' \in \Ind_{\ell'}} M^{\gamma'} X^{\otimes}_{\gamma'}\) and \(L_0 = \sum_{\gamma \in \Ind_{\ell}} L_0^{\gamma} X^{\otimes}_{\gamma}\).

The point of this proposition is that given \(A(D)\), typically it is not too difficult, by looking at the Euclidean analog for instance, to come up with an \(L_0 \in \End(E;F) \otimes T_{\ell}(\g_1)\) with \(\operatorname{Sym}(L_0)(D)\) cocanceling such that \(L_0(D) \circ A(D)\) is almost zero (in the sense that it involves a lot of commutators). We will then apply some \(M(D)\) that satisfies the conditions of Proposition~\ref{prop:cocan_preserve}, to the composition \(L_0(D) \circ A(D)\), hoping that we have
\begin{equation*}
(M \circ L_0 - N)(D) \circ A(D) = 0
\end{equation*} 
for some \(N(D)\) with \(\operatorname{Sym}(N)(D) = 0\). Then we may apply Theorem~\ref{thm:can_est} and Theorem~\ref{thm:Hardy} with \(L(D) \defeq (M \circ L_0 - N)(D)\), because Proposition~\ref{prop:cocan_preserve} guarantees that \(\operatorname{Sym}(L)(D) = 0\). See examples in Section~\ref{sect:applications}.

The proof of Proposition~\ref{prop:cocan_preserve} relies on the following lemma:

\begin{lemma} \label{lem:cocan_open}
Suppose \(L \in \End(E;F) \otimes T_{\ell}(\g_1)\). Then \(\operatorname{Sym}(L)(D)\) is cocanceling, if and only if 
\begin{equation*}
\bigcap_{\xi \in U} \operatorname{ker} \, \operatorname{Sym}(L)(\xi) = \{0\}
\end{equation*}
for any non-empty open subset \(U\) of \(\R^m\).
\end{lemma}

\begin{proof}
Suppose \(\operatorname{Sym}(L)(D)\) is cocanceling, and \(U\) is a non-empty open subset of \(\R^m\). Then there exist vectors \(v_1, \dotsc, v_{m+\ell-1} \in \R^m\) such that the following holds:
\begin{enumerate}[(i)]
\item any \(m\) distinct vectors from \(v_1, \dotsc, v_{m+\ell-1}\) are linearly independent;
\item for any set \(S\) of \(\ell\) distinct vectors from \(v_1, \dotsc, v_{m+\ell-1}\), the orthogonal complement of the \(m-1\) vectors in \(\{v_1, \dotsc, v_{m+\ell-1}\} \setminus S\) contains some \(\xi_S \in U\).
\end{enumerate}
Let \(\Lambda\) be the collection of all sets of \(\ell\) distinct vectors from \(v_1, \dotsc, v_{m+\ell-1}\). For \(S \in \Lambda\), let's write \(v_{S} \cdot \xi\) as a shorthand for \(\prod_{v \in S} v \cdot \xi\). Then \(\{v_{S} \cdot \xi\}_{S \in \Lambda}\) is a basis of \(S_{\ell}(\g_1)\). Hence \(\operatorname{Sym}(L)(\xi)\) can be expanded as \begin{equation*}\operatorname{Sym}(L)(\xi) = \sum_{S \in \Lambda} C_S (v_{S} \cdot \xi)\end{equation*} where each \(C_S \in \End(E;F)\). By setting \(\xi = \xi_S\), we see that \(C_S = \operatorname{Sym}(L)(\xi_S)\), so 
\begin{equation*}
\operatorname{Sym}(L)(\xi) = \sum_{S \in \Lambda} \operatorname{Sym}(L)(\xi_S) (v_{S} \cdot \xi).
\end{equation*} 
This shows that 
\begin{equation*}
\bigcap_{S \in \Lambda} \operatorname{ker} \, \operatorname{Sym}(L)(\xi_S) \subset \bigcap_{\xi \in \R^m} \operatorname{ker} \, \operatorname{Sym}(L)(\xi),
\end{equation*}
which is \(\{0\}\) since \(\operatorname{Sym}(L)(D)\) is cocanceling. It follows that 
\begin{equation*}
\bigcap_{\xi \in U} \operatorname{ker} \, \operatorname{Sym}(L)(\xi) = \{0\}.
\end{equation*}
The converse is obvious.
\end{proof}

\begin{proof}[Proof of Proposition~\ref{prop:cocan_preserve}]
Suppose \(\operatorname{Sym}(L_0)(D)\) is cocanceling, and suppose there exists \(\xi_0 \in \R^m\) such that \(\operatorname{Sym}(M)(\xi_0)\) is injective. Then there exists an open neighborhood \(U\) of \(\xi_0\) such that \(\operatorname{Sym}(M)(\xi)\) is injective for all \(\xi \in U\). Now \begin{equation*}\operatorname{Sym}(M \circ L_0)(\xi) = \operatorname{Sym}(M)(\xi) \circ \operatorname{Sym}(L_0)(\xi)\end{equation*} for all \(\xi \in \R^m\), and
\begin{equation*}
\operatorname{ker} \, [\operatorname{Sym}(M)(\xi) \circ \operatorname{Sym}(L_0)(\xi)] = \operatorname{ker} \, \operatorname{Sym}(L_0)(\xi)
\end{equation*} 
whenever \(\xi \in U\).
As a result, 
\begin{align*}
\bigcap_{\xi \in \R^m} \operatorname{ker} \, \operatorname{Sym}(M \circ L_0)(\xi) 
&\subset \bigcap_{\xi \in U} \operatorname{ker} \, [\operatorname{Sym}(M)(\xi) \circ \operatorname{Sym}(L_0)(\xi)] \\
&= \bigcap_{\xi \in U} \operatorname{ker} \, \operatorname{Sym}(L_0)(\xi)
= \{0\},
\end{align*} 
the last equality following from Lemma~\ref{lem:cocan_open}. Hence \(\operatorname{Sym}(M \circ L_0)(D)\) is cocanceling.
\end{proof}

\section{Applications} \label{sect:applications}

We are now ready to revisit the Examples~\ref{eg1}, \ref{eg2} and \ref{eg3}. As usual, \(X_1, \dots, X_m\) represents a basis of \(\g_1\).

Much of the following proposition is not new. We include it here mainly for the purpose of exposition, to illustrate how our methods apply in this simple case.

\begin{proposition} \label{prop:subelliptic_gradient}
Suppose \(G\) is a stratified homogeneous group with homogeneous dimension \(Q \ge 2\). Let \(u \in C^{\infty}_c(G;\R)\). Then
\begin{equation} \label{eq:conclude1a}
\|u\|_{L^{\frac{Q}{Q-1}}(G;\R)} \leq C \sum_{j=1}^m \|X_j u\|_{L^1(G;\R)}.
\end{equation}
Furthermore, 
\begin{equation} \label{eq:conclude1b}
\int_G \frac{|u(x)|}{\|x\|} dx \leq C \sum_{j=1}^m \|X_j u\|_{L^1(G;\R)},
\end{equation}
and more generally
\begin{equation} \label{eq:conclude1c}
\left( \int_G \bigl(\|x\|^{Q-1} |u(x)| \bigr)^p \frac{dx}{\|x\|^Q} \right)^{1/p} \leq C \sum_{j=1}^m \|X_j u\|_{L^1(G;\R)}
\end{equation}
for all \(p \in [1, \frac{Q}{Q-1})\).
\end{proposition}

\begin{proof}
We will use the notations from Example \ref{eg1}. Additionally, let \(F = \Lambda^2(\g_1^*)\). 
Define \(L_0 \in \End(E;F) \otimes T_1(\g_1)\) such that 
\begin{equation*}
L_0(D) f \defeq \sum_{1 \leq i < j \leq m} (X_i f_j - X_j f_i) e^i \wedge e^j
\end{equation*} 
for \(f = \sum_{1 \leq j \leq m} f_j e^j \in C^{\infty}(G;E)\). Then 
\begin{equation*}
L_0 (D) \circ A(D) u = \sum_{1 \leq i < j \leq m} [X_i,X_j] u \, e^i \wedge e^j
\end{equation*} 
if \(u \in C^{\infty}(G;V)\).
Furthermore, define \(M \in \End(F;F) \otimes T_r(\g_1)\) such that \begin{equation*}
M(D) g \defeq \sum_{1 \leq i < j \leq m} X_j^r g_{ij} \, e^i \wedge e^j\end{equation*} 
for \(g = \sum_{1 \leq i < j \leq m} g_{ij} \, e^i \wedge e^j\); here \(r\) is the step of the Lie algebra \(\g\), so that any commutator of length \(r+1\) of elements from \(\g_1\) is zero.
Then there exists \(N \in \End(E;F) \otimes T_{r+1}(\g_1)\) such that \begin{equation*}
 (M \circ L_0 - N)(D)  \circ A(D) = 0\end{equation*} and \begin{equation*}\operatorname{Sym}(N)(D) = 0.\end{equation*} This holds because 
\begin{align*}
X_j^r [X_i, X_j] 
&= X_j^{r-1} [X_i, X_j] X_j + X_j^{r-1} [X_j, [X_i,X_j]] \\
&= X_j^{r-1} [X_i, X_j] X_j + X_j^{r-2} [X_j, [X_i,X_j]] X_j + X_j^{r-2} [X_j, [X_j, [X_i,X_j]]] \\
&= \dotsb \\
&= \sum_{s=1}^r X_j^{r-s} \underbrace{ [X_j, [X_j, \cdots [X_i, X_j] ] ] }_{\text{\(s\) brackets}} X_j + \underbrace{ [X_j, [X_j, \cdots [X_i, X_j] ] ] }_{\text{\(r+1\) brackets}}
\end{align*}
while the last term is zero since it has \(r+1\) brackets; hence it suffices to take \begin{equation*}N(D) f = \sum_{s=1}^r X_j^{r-s} \underbrace{ [X_j, [X_j, \cdots [X_i, X_j] ] ] }_{\text{\(s\) brackets}} f_j \, e^i \wedge e^j\end{equation*} if \(f = \sum_{j=1}^m f_j \, e^j\). 
Since \(\operatorname{Sym}(M)(\xi_0)\) is the identity map on \(F\) when \(\xi_0 = (1,\dotsc,1)\), and \(\operatorname{Sym}(L_0)(D)\) is cocanceling (here we use \(m \geq 2\) which follows from the assumption \(Q \geq 2\)), by Proposition~\ref{prop:cocan_preserve}, we have \(\operatorname{Sym} (M \circ L_0 - N)(D)\) being cocanceling. 
Taking \(L \defeq M \circ L_0 - N\), \eqref{eq:conclude1a} now follows from Theorem~\ref{thm:can_est}, and \eqref{eq:conclude1b}, \eqref{eq:conclude1c} follow from Theorem~\ref{thm:Hardy}.
\end{proof}

Next we generalize the previous proposition to a `higher order gradient'.

\begin{proposition} \label{prop:subelliptic_highergradient}
Suppose \(G\) is a stratified homogeneous group with homogeneous dimension \(Q \ge 2\). Let \(k \in \N\) and \(u \in C^{\infty}_c(G;\R)\). Then
\begin{equation} \label{eq:conclude2a}
\|D^{k-1} u\|_{L^{\frac{Q}{Q-1}}(G;\R)} \leq C \sum_{j=1}^m \|X_j^k u\|_{L^1(G;\R)},
\end{equation}
and for all \(\ell \in \{1, \dotsc, \min\{k,Q-1\}\}\), we have
\begin{equation} \label{eq:conclude2b}
\int_G \frac{|D^{k-\ell} u(x)|}{\|x\|^{\ell}} dx \leq C \sum_{j=1}^m \|X_j^k u\|_{L^1(G;\R)}
\end{equation}
and more generally
\begin{equation} \label{eq:conclude2c}
\left( \int_G (\|x\|^{Q-\ell} |D^{k-\ell} u(x)|)^p \frac{dx}{\|x\|^Q} \right)^{1/p} \leq C \sum_{j=1}^m \|X_j^k u\|_{L^1(G;\R)}
\end{equation}
for all \(p \in [1, \frac{Q}{Q-1})\). 
If \(k \geq Q\), we also have
\begin{equation} \label{eq:conclude2d}
\|D^{k-Q} u\|_{L^{\infty}(G;\R)} \leq C \sum_{j=1}^m \|X_j^k u\|_{L^1(G;\R)}.
\end{equation}
\end{proposition}

\begin{proof}
We adopt the notations from Example~\ref{eg2}. Additionally, let \(F = \Lambda^2(\g_1^*)\). 
Define \(L_0 \in \End(E;F) \otimes T_k(\g_1)\) such that 
\begin{equation*}
L_0(D) f 
\defeq 
\sum_{1 \leq i < j \leq m} (X_i^k f_j - X_j^k f_i) e^i \wedge e^j
\end{equation*} 
for \(f = \sum_{1 \leq j \leq m} f_j e^j \in C^{\infty}(G;E)\). 
Then 
\begin{equation*}
L_0 (D) \circ A(D) u = \sum_{1 \leq i < j \leq m} [X_i^k,X_j^k] u \, e^i \wedge e^j,
\end{equation*}
or more explicitly
\begin{equation*}
L_0(D) \circ A(D) u = \sum_{1 \leq i < j \leq m} \sum_{1 \leq s, t \leq k} X_i^{k-s} X_j^{k-t} [X_i,X_j] X_j^{t-1} X_i^{s-1} u \, e^i \wedge e^j
\end{equation*} 
if \(u \in C^{\infty}(G;V)\).
Furthermore, define \(M \in \End(F;F) \otimes T_{k^2 r}(\g_1)\) such that \begin{equation*}
M(D) g \defeq \sum_{1 \leq i < j \leq m} X_j^{k^2 r} g_{ij} \, e^i \wedge e^j
\end{equation*} 
for \(g = \sum_{1 \leq i < j \leq m} g_{ij} \, e^i \wedge e^j\); again \(r\) is the step of the Lie algebra \(\g\).
Then there exists \(N \in \End(E;F) \otimes T_{k(kr+1)}(\g_1)\) such that \begin{equation*}
(M \circ L_0 - N)(D)  \circ A(D) = 0
\end{equation*} and 
\begin{equation*}
    \operatorname{Sym}(N)(D) = 0.
\end{equation*} 
This holds because 
\begin{align*}
X_j^{k^2 r} [X_i^k, X_j^k] 
&= \sum_{s=1}^{kr} X_j^{k(kr-s)} \underbrace{ [X_j^k, [X_j^k, \cdots [X_i^k, X_j^k] ] ] }_{\text{\(s\) brackets}} X_j^k + \underbrace{ [X_j^k, [X_j^k, \cdots [X_i^k, X_j^k] ] ] }_{\text{\(k r + 1\) brackets}}.
\end{align*}
We claim that the last term is zero: this is because there are \(k r + 1\) brackets, each of which involving at least one \(X_i\)'s, but there are only \(k\) \(X_i\)'s; hence at least one of the \(X_i\)'s is in a bracket of length at least \(r+1\), which is zero. As a result, it suffices to take \begin{equation*}
N(D) f \defeq \sum_{s=1}^{kr} X_j^{k(kr-s)} \underbrace{ [X_j^k, [X_j^k, \cdots [X_i^k, X_j^k] ] ] }_{\text{\(s\) brackets}} X_j^k f_j \, e^i \wedge e^j
\end{equation*} 
if \(f = \sum_{j=1}^m f_j \, e^j\). 
Since \(\operatorname{Sym}(M)(\xi_0)\) is the identity map on \(F\) when \(\xi_0 = (1,\dotsc,1)\), and \(\operatorname{Sym}(L_0)(D)\) is cocanceling  (here we use \(m \geq 2\) which follows from the assumption \(Q \geq 2\)), by Proposition~\ref{prop:cocan_preserve}, we have \(\operatorname{Sym} (M \circ L_0 - N)(D)\) being cocanceling. 
Taking \(L \defeq M \circ L_0 - N\), \eqref{eq:conclude2a} now follows from Theorem~\ref{thm:can_est}, and \eqref{eq:conclude2b}, \eqref{eq:conclude2c} and \eqref{eq:conclude2d} follow from Theorem~\ref{thm:Hardy}.
\end{proof}

Finally, on a general stratified homogeneous group \(G\) with \(G \ne \R\), we have the following endpoint Korn--Sobolev inequality, and the following endpoint Korn--Hardy inequality.

\begin{theorem} \label{thm:Korn_Sobolev}
Suppose \(G\) is a stratified homogeneous group with homogeneous dimension~\(Q \ge 2\). Let \(u_1, \dots, u_m \in C^{\infty}_c(G;\R)\) where \(m = \dim \g_1\). Then 
\begin{equation} \label{eq:conclude3a}
\sum_{j=1}^m \|u_j\|_{L^{\frac{Q}{Q-1}}(G;\R)} \leq C \sum_{i,j=1}^m \|X_i u_j + X_j u_i\|_{L^1(G;\R)}.
\end{equation}
Furthermore, 
\begin{equation} \label{eq:conclude3b}
\sum_{j=1}^m \int_G \frac{|u_j(x)|}{\|x\|} dx \leq C  \sum_{i,j=1}^m  \|X_i u_j + X_j u_i\|_{L^1(G;\R)},
\end{equation}
and more generally
\begin{equation} \label{eq:conclude3c}
\sum_{j=1}^m \left( \int_G (\|x\|^{Q-1} |u_j(x)|)^p \frac{dx}{\|x\|^Q} \right)^{1/p} \leq C  \sum_{i,j=1}^m \|X_i u_j + X_j u_i\|_{L^1(G;\R)}
\end{equation}
for all \(p \in [1, \frac{Q}{Q-1})\).
\end{theorem}

\begin{proof}
We adopt the notations from Example~\ref{eg3}. Additionally, let \(F = \Lambda^2(\g_1)\) and define \(L_0 \in \End(E;F) \otimes T_2(\g_1)\) such that
\begin{equation*}
L_0(D) f \defeq \sum_{1 \leq i < j \leq m} \left( \frac{ X_i^2 f_{jj} + X_j^2 f_{ii} }{2} - X_i X_j f_{ij} \right) \, e^i \wedge e^j
\end{equation*}
if \(f = \sum_{1 \leq i \leq j \leq m} f_{ij} \, e^{ij}\). Then
\begin{equation*}
L_0(D) \circ A(D) u = \sum_{1 \leq i < j \leq m} (X_i [X_i,X_j] u_j + [X_j^2,X_i] u_i) \, e^i \wedge e^j
\end{equation*}
if \(u = \sum_{1 \leq j \leq m} u_j e^j\). Furthermore, define \(M \in \End(F;F) \otimes T_{r(2r+4)}(\g_1)\) such that 
\begin{equation*}
M(D) g = \sum_{1 \leq i < j \leq m} X_i^{r(2r+2)} X_j^{2r} g_{ij} \, e^i \wedge e^j
\end{equation*}
for \(g = \sum_{1 \leq i < j \leq m} g_{ij} \, e^i \wedge e^j\); again \(r\) is the step of the Lie algebra \(\g\). Then there exists \(N \in \End(E;F) \otimes T_{r(2r+4)+2}\) such that
\begin{equation*}
(M \circ L_0 - N)(D) \circ A(D) = 0
\end{equation*}
and
\begin{equation*}
\operatorname{Sym}(N)(D) = 0.
\end{equation*}
This holds because 
\begin{equation*}
X_j^{2r} X_i [X_i, X_j]
= \sum_{s = 1}^{2r} X_j^{2r-s} \underbrace{ [X_j, [X_j, \cdots,  X_i [X_i, X_j] ] ] }_{\text{\(s\) brackets}} X_j +  \underbrace{ [X_j, [X_j, \cdots,  X_i [X_i, X_j] ] ] }_{\text{\(2r+1\) brackets}},
\end{equation*}
the last term being zero because there are \(2r+1\) brackets, but only two \(X_i\)'s; also
\begin{align*}
X_i^{r(2r+2)} X_j^{2r} [X_j^2, X_i] &= \sum_{s = 1}^{r(2r+2)} X_i^{r(2r+2)-s} \underbrace{ [X_i, [X_i, \cdots, X_j^{2r} [X_j^2, X_i] ] ] }_{\text{\(s\) brackets}} X_i \\
& \quad + \underbrace{ [X_i, [X_i, \cdots, X_j^{2r} [X_j^2, X_i] ] ] }_{\text{\(r(2r+2)+1\) brackets}},
\end{align*}
the last term being zero because there are \(r(2r+2)+1\) brackets, but only \(2r+2\) \(X_j\)'s.
As a result, it suffices to take
\begin{align*}
N(D) f 
&\defeq \frac{1}{2} \sum_{1 \leq i < j \leq m} \sum_{s = 1}^{2r} X_i^{r(2r+2)} X_j^{2r-s} \underbrace{ [X_j, [X_j, \cdots,  X_i [X_i, X_j] ] ] }_{\text{\(s\) brackets}} f_{jj}  \, e^i \wedge e^j \\
& \quad +  \frac{1}{2} \sum_{1 \leq i < j \leq m} \sum_{s = 1}^{r(2r+2)} X_i^{r(2r+2)-s} \underbrace{ [X_i, [X_i, \cdots, X_j^{2r} [X_j^2, X_i] ] ] }_{\text{\(s\) brackets}} f_{ii} \, e^i \wedge e^j
\end{align*}
if \(f = \sum_{j=1}^m f_j \, e^j\). Since \(\operatorname{Sym}(M)(\xi_0)\) is the identity map on \(F\) when \(\xi_0 = (1,\dotsc,1)\), and \(\operatorname{Sym}(L_0)(D)\) is cocanceling  (here we use \(m \geq 2\) which follows from the assumption \(Q \geq 2\)), by Proposition~\ref{prop:cocan_preserve}, we have \(\operatorname{Sym} (M \circ L_0 - N)(D)\) being cocanceling. Taking \(L \defeq M \circ L_0 - N\), \eqref{eq:conclude3a} follows from Theorem~\ref{thm:can_est}, and \eqref{eq:conclude3b}, \eqref{eq:conclude3c} follow from Theorem~\ref{thm:Hardy}.
\end{proof}

\appendix

\section{Proof of Theorem~\ref{thm:kernel_prelim}}

In this appendix we prove Theorem~\ref{thm:kernel_prelim}. For this we will need to work with (isotropic \(L^2\)-based) Sobolev spaces \(H^s(G)\), defined for \(s \in \R\) to be the set of all real tempered distributions \(u\) on the underlying Euclidean space, for which
\begin{equation*}
\|u\|_{H^s(G)} \defeq \int_G (1+ 4 \pi^2 |\xi|^2)^{s} |\mathcal{F} u(\xi)|^2 \,d\xi < \infty;
\end{equation*}
here \(\mathcal{F} u\) is the (Euclidean) Fourier transform of \(u\), which is required to agree with a locally integrable function if \(u \in H^s(G)\), where we use the convention
\begin{equation*}
\mathcal{F} \phi (\xi) = \int_{G} e^{-2\pi \imath \xi \cdot x} \phi (x)\,dx.
\end{equation*}

The space \(H^s(G)\) is a real Hilbert space under the real inner product 
\begin{equation*}
(u,v)_{H^s(G)} \defeq \int_G (1+ 4 \pi^2 |\xi|^2)^{s} \mathcal{F} u(\xi) \mathcal{F} v(-\xi) d\xi.
\end{equation*}
For every \(s \in \R\), the dual space to \(H^s(G)\) is \(H^{-s}(G)\), via the pairing
\begin{equation*}
( u, \phi ) \defeq \int_G \mathcal{F} u(\xi) \mathcal{F} \phi(-\xi) d\xi
\end{equation*}
if \(u \in H^{-s}(G)\) and \(\phi \in H^s(G)\). 

We are now ready to prove Theorem~\ref{thm:kernel_prelim}.

\begin{proof}[Proof of Theorem~\ref{thm:kernel_prelim}]
First, the hypoellipticity of \(A(D)\) will be used to show that \(A^t(D)\) is locally solvable (cf.\ Tr\`{e}ves \cite{Treves}*{Theorem 52.2}): in particular, there exists a bounded open set \(\Omega \subset G\) containing \(0\) and some \(\K_{\text{loc}} \in \D'(G;\End(V;E))\) such that 
\begin{equation*}
A^t(D) \K_{\text{loc}} = \delta \otimes I \quad \text{on \(\Omega\)},
\end{equation*}
where we have extended by tensor product with \(V\) to obtain a continuous linear operator \(A^t(D) \colon \D'(G;\End(V;E)) \to \D'(G;\End(V;V)) \).

Indeed, let \(K\) be any compact subset of \(G\). We claim that for every \(s \in \N\) there exists \(s' \in \N\) and \(C > 0\) such that 
\begin{equation} \label{eq:claim1_hypo}
\|\phi \|_{H^s(G;V)} \leq C ( \|A(D) \phi\|_{H^{s'}(G;E)} + \|\phi\|_{L^2(G;V)} )
\end{equation}
for all \(\phi \in \D(K;V)\). 
To prove this, let \(\tau\) be the Fr\'{e}chet topology we endowed on \(\D(K;V)\), and \(\tilde{\tau}\) be the locally convex metrizable topology on \(\D(K;V)\) given by a countable family of separating seminorms \(\{\vertiii{\cdot}_s \colon s \in \N_0\}\) where 
\begin{equation*}
\vertiii{\phi}_s \defeq \|A(D) \phi \|_{H^s(G;E)} + \|\phi\|_{L^2(G;V)}
\end{equation*}
for \(\phi \in \D(K;V)\). 
One can check that \(\D(K;V)\) is complete under the topology \(\tilde{\tau}\), making \((\D(K;V),\tilde{\tau})\) a Fr\'{e}chet space as well. 
This is because if \((\phi_i)_{i \in \N}\) is a Cauchy sequence in \((\D(K;V), \tilde{\tau})\), then there exists \(\phi \in L^2(G;V)\) and \(\Phi \in \bigcap_{s \in \N_0} H^s(G;E)\) such that
\begin{equation*}
\lim_{i \to \infty} \|\phi_i - \phi\|_{L^2(G;V)} = 0 \quad \text{and} \quad \lim_{i \to \infty} \|A(D) \phi_i - \Phi \|_{H^s(G;E)} = 0
\end{equation*}
for every \(s \in \N_0\). 
In particular, the case \(s = 0\) shows that \(A(D) \phi = \Phi\) in \(\D'(G;E)\), and Sobolev embedding implies \(\Phi \in C^{\infty}(G;E)\), so the hypoellipiticity of \(A(D)\) implies that \(\phi \in C^{\infty}(G;V)\), and the support conditions on the sequence \(\phi_i\) further implies \(\phi \in \D(K;V)\). 
Furthermore, \(\vertiii{\phi_i - \phi}_s = \|A^t(D) \phi_i - \Phi\|_{H^s(G;E)} + \|\phi_i - \phi\|_{L^2(G;V)} \to 0\) as \(i \to +\infty\), for every \(s \in \N_0\). 
Thus the sequence \((\phi_i)_{i \in \N_0}\) converges in \((\D(K;V),\tilde{\tau})\), and this verifies the completeness of \(\D(K;V)\) under the topology \(\tilde{\tau}\).

Since for every \(s \in \N_0\), there exists \(n \in \N_0\) and \(c > 0\) such that 
\begin{equation*}
\vertiii{\phi}_s \leq c \|\phi\|_{C^n(K;V)} 
\end{equation*}
for all \(\phi \in \D(K;V)\), the identity map is a continuous linear map from \((\D(K;V),\tau)\) to \((\D(K;V),\tilde{\tau})\). Since both \((\D(K;V),\tau)\) and \((\D(K;V),\tilde{\tau})\) are Fr\'{e}chet spaces, the open mapping theorem implies that the identity map is also continuous from \((\D(K;V),\tilde{\tau})\) to \((\D(K;V),\tau)\), which implies (cf.\ \cite{Grubb}*{Lemma B.7}) that for every \(n \in \N_0\), there exists \(s' \in \N_0\) and \(c' > 0\) such that
\begin{equation*}
\|\phi\|_{C^n(K;V)} \leq c' \vertiii{\phi}_{s'}
\end{equation*}
for all \(\phi \in \D(K;V)\). This in turn establishes our previous claim \eqref{eq:claim1_hypo}, because 
\begin{equation*}
\|\phi\|_{H^s(G;V)} \leq c_s \|\phi\|_{C^s(K;V)}
\end{equation*}
for all \(\phi \in \D(K;V)\) and \(s \in \N_0\).

Now fix \(s \in \N\) large enough so that \(\delta \in H^{-s}(G)\). Let \(s'\) and \(C\) be chosen as in the claim \eqref{eq:claim1_hypo} depending on \(s\). If \(\Omega \subset G\) is an open neighborhood of \(0\) with sufficiently small diameter, we claim that 
\begin{equation}
\label{eq_ab5rah6phuY0keirai4ahb1u}
C \|\phi\|_{L^2(G;V)} \leq \frac{1}{2} \|\phi \|_{H^s(G;V)}
\end{equation}
for all \(\phi \in \D(\overline{\Omega};V)\).
Indeed, by the fundamental theorem of calculus, or equivalently, by Poincar\'e's inequality, for all \(\phi \in \D(\overline{\Omega};V)\), we have
\begin{equation*}
\|\phi\|_{L^2(G;V)} \leq (  \operatorname{diam}\, \Omega) \|\phi\|_{H^1(G;V)},
\end{equation*}
so it suffices to take \(\operatorname{diam}\, \Omega \leq 1/(2C)\) for the desired inequality to hold. 
In this case, it follows from \eqref{eq:claim1_hypo} and \eqref{eq_ab5rah6phuY0keirai4ahb1u} that 
\begin{equation*}
\|\phi \|_{H^s(G;V)} \leq 2C \|A(D) \phi\|_{H^{s'}(G;E)}
\end{equation*}
for all \(\phi \in \D(\overline{\Omega};V)\). Let \(Y\) be the closure, in \(H^{s'}(G;E)\), of the set 
\begin{equation*}
\{ A(D) \phi \in H^{s'}(G;E) \colon \phi \in \D(\overline{\Omega};V) \}.
\end{equation*}
By density, the map \(A(D) \phi \mapsto \phi\) can be extended to a continuous linear map of \(Y\) into \(H^s(G;V)\). It can be further extended as a continuous linear map \(T \colon H^{s'}(G;E) \to H^s(G;V)\), by setting it to be zero on the orthogonal complement of \(Y\) in \(H^{s'}(G;E)\). The adjoint \(T^*\) is then a continuous linear map \(T^* \colon H^{-s}(G;V) \to H^{-s'}(G;E)\), such that for all \(u \in H^{-s}(G;V)\), one has
\begin{equation*}
A^t(D) T^* u = u \quad \text{on \(\Omega\)};
\end{equation*}
indeed, for every \(\phi \in \D(\overline{\Omega};V)\), we have
\(
\langle A^t(D) T^* u, \phi \rangle_{V,G} 
= \langle u, T A(D) \phi \rangle_{V,G}
= \langle u, \phi \rangle_{V,G}.
\)
It remains to observe that if \(\{v_i\}_{i = 1}^{\dim V}\) is an orthonormal basis for \(V\), then 
\begin{equation*}
\K_{\text{loc}} \defeq \sum_{i = 1}^{\dim V} T^*(\delta \otimes v_i) \otimes v_i
\end{equation*}
is in \(H^{-s'}(G;E) \otimes V = H^{-s'}(G;\End(V;E)) \subset \D'(G;\End(V;E))\) and 
\begin{equation*}
A^t(D) \K_{\text{loc}} = \sum_{i=1}^{\dim V} \delta \otimes v_i \otimes v_i = \delta \otimes I \quad \text{on \(\Omega\)},
\end{equation*} 
as desired.

Next, using the assumption that \(A^t(D)\) is hypoelliptic and homogeneous of order \(k\), we may use a rescaling of \(\K_{\text{loc}}\) to construct a global \(\K \in \D'(G;\End(V;E))\) such that 
\begin{equation*}
A^t(D) \K = \delta \otimes I \quad \text{on \(G\)}.
\end{equation*}
Indeed, since \(A^t(D)\) is hypoelliptic and \(A^t(D)\K_{\text{loc}} = 0\) on \(\Omega \setminus \{0\}\), we have 
\begin{equation*}
\K_{\text{loc}} \in C^{\infty}(\Omega \setminus \{0\};\End(V;E)).
\end{equation*} 
Let \(\eta \in C^{\infty}_c(\Omega;\R)\) be such that \(\eta(x) = 1\) in an open set containing \(0\). Then letting \(\K^{(1)} \defeq \eta \K_{\text{loc}}\), we have
\begin{equation*}
A^t(D) \K^{(1)} = \delta \otimes I + \Phi^{(1)} \quad \text{on \(G\)}
\end{equation*}
where \(\Phi^{(1)} \in C^{\infty}(G;\End(V;V))\) vanishes in an open set containing \(0\). Now for \(\lambda > 0\), let
\begin{equation*}
\K^{(\lambda)} \defeq \lambda^{k-Q} \K^{(1)} \circ \delta_{\lambda^{-1}} \quad \text{and} \quad \Phi^{(\lambda)} \defeq \lambda^{k-Q} \Phi \circ \delta_{\lambda^{-1}}
\end{equation*}
so that 
\begin{equation*}
A^t(D) \K^{(\lambda)} = \delta \otimes I + \Phi^{(\lambda)} \quad \text{on \(G\)}
\end{equation*}
for every \(\lambda > 0\). As \(\lambda \to +\infty\), \(\Phi^{(\lambda)} \to 0\) in the topology of \(\D'(G;\End(V;V))\) because \(\Phi\) vanishes in an open set containing \(0\). We claim that 
\begin{enumerate}[(i)]
\item if \(k < Q\), then \(\K^{(\lambda)}\) converges in the topology of \(\D'(G;\End(V;E))\) as \(\lambda \to +\infty\);
\item if \(k \geq Q\), then there exist \(\End(V;E)\)--valued polynomials \(p_0, p_1, \dotsc, p_{k-Q}\) on \(G\), with \(p_i\) homogeneous of degree \(i\) for all \(i \in \{0, \dotsc, k - Q\}\), such that 
\begin{equation} \label{eq:tildeKlam_def}
\tilde{\K}^{(\lambda)} \defeq \K^{(\lambda)} - (\log \lambda) p_{k-Q}  - \sum_{i=1}^{k-Q}  \lambda^{i} p_{k-Q-i}
\end{equation}
converges in the topology of \(\D'(G;\End(V;E))\) as \(\lambda \to +\infty\). 
\end{enumerate}
In the first case, we define \(\K \in \D'(G;\End(V;E))\) to the limit of \(\K^{(\lambda)}\) as \(\lambda \to +\infty\); in the second case, we define \(\K \in \D'(G;\End(V;E))\) to be the limit of \(\tilde{\K}^{(\lambda)}\) as \(\lambda \to +\infty\). We then have 
\begin{equation*}
A^t(D) \K = \delta \otimes I \quad \text{on \(G\)}
\end{equation*}
in either case, as desired.

To prove the claims above, first observe that 
\begin{equation*}
\K' \defeq \lim_{\lambda \to 1} \frac{\K^{(\lambda)} - \K^{(1)}}{\lambda - 1}
\end{equation*}
exists in \(\D'(G;\End(V;E))\). This is because for any \(\phi \in \D(G;\End(V;E))\), 
\begin{equation*}
\Bigl\langle \frac{\K^{(\lambda)} - \K^{(1)}}{\lambda - 1}, \phi \Bigr\rangle
= \Bigl\langle \K^{(1)}, \frac{\lambda^{k} \phi \circ \delta_{\lambda} - \phi }{\lambda - 1} \Bigr\rangle
\end{equation*}
and 
\begin{equation*}
\frac{\lambda^{k} \phi \circ \delta_{\lambda} - \phi }{\lambda - 1}
\end{equation*}
converges in the topology of \(\D(G;\End(V;E))\) as \(\lambda \to 1\) (indeed, recalling our notation \(x = (x_1, \dotsc, x_r) \in \g =  \g_1 \oplus \dotsb \oplus \g_r\), we have
\begin{equation*}
\frac{\phi \circ \delta_{\lambda} - \phi }{\lambda - 1} \to \sum_{j=1}^r j \, x_j \cdot \partial_{x_j} \phi(x) 
\end{equation*}
in the topology of \(\D(G;\End(V;E))\) as \(\lambda \to 1\), which shows that 
\begin{equation*}
\frac{\lambda^{k} \phi \circ \delta_{\lambda} - \phi }{\lambda - 1} \to k \phi(x) + \sum_{j=1}^r j x_j \cdot \partial_{x_j} \phi(x) 
\end{equation*}
in the topology of \(\D(G;\End(V;E))\) as \(\lambda \to 1\)). Furthermore, \(\K'\) is compactly supported (since \(\K^{(1)}\) is), and \(A^t(D) \K' \in C^{\infty}(G;\End(V;V))\) (indeed, by the continuity of \(A^t(D) \colon \D'(G;\End(V;E)) \to \D'(G;\End(V;V))\),
\begin{equation*}
\begin{split}
A^t(D) \K' &= \lim_{\lambda \to 1} A^t(D) \left( \frac{ \lambda^{k-Q} \K^{(1)} \circ \delta_{\lambda^{-1}} - \K^{(1)} }{\lambda - 1} \right) \\
&= \lim_{\lambda \to 1} \frac{ \lambda^{-Q} (A^t(D) \K^{(1)}) \circ \delta_{\lambda^{-1}} - (A^t(D) \K^{(1)})}{\lambda - 1}\\
&= \lim_{\lambda \to 1} \frac{ \lambda^{-Q} \Phi^{(1)} \circ \delta_{\lambda^{-1}} - \Phi^{(1)}}{\lambda - 1}\\
& = -Q \Phi^{(1)} - \sum_{j=1}^r j x_j \cdot \partial_{x_j} \Phi^{(1)} (x) 
\end{split}
\end{equation*}
which is in \(C^{\infty}(G;\End(V;V))\)), so the hypoellipticity of \(A^t(D)\) implies that we have  \(\K' \in C^{\infty}(G;\End(V;E))\)). It follows that \(\K' \in \D(G;\End(V;E))\). Now for every \(\lambda > 0\), the derivative \(\frac{d}{d \lambda} \K^{(\lambda)}\) exists in \(\D'(G;\End(V;E))\) and is equal to \(\lambda^{k-Q-1} \K' \circ \delta_{\lambda^{-1}}\), because 
\begin{equation*}
\frac{1}{\lambda} \lim_{s \to 1} \frac{\K^{(\lambda s)} - \K^{(\lambda)}}{s-1} = \frac{1}{\lambda} \lambda^{k-Q} \lim_{s \to 1} \frac{s^{k-Q} \K^{(1)} \circ \delta_{s^{-1}} - \K^{(1)}}{s-1} \circ \delta_{\lambda^{-1}}
\end{equation*}
converges to \(\lambda^{k-Q-1} \K' \circ \delta_{\lambda^{-1}}\) in \(\D'(G;\End(V;E))\). We now consider two cases. If \(k < Q\), then for every \(\phi \in \D(G;\End(V;E))\), we have
\begin{equation*}
\begin{split}
 \int_1^{\infty} \Bigl\lvert\Bigl\langle \frac{d}{d \lambda} \K^{(\lambda)}, \phi \Bigr\rangle \Bigr\rvert\, d\lambda 
=& \int_1^{\infty} \lambda^{k-Q-1} \left| \langle \K' \circ \delta_{\lambda^{-1}}, \phi \rangle \right| d\lambda \\
\leq & \|\K'\|_{L^{\infty}(G;\End(V;E))} \|\phi\|_{L^1(G;\End(V;E))} \int_1^{\infty} \lambda^{s-Q-1} d\lambda < \infty.
\end{split}
\end{equation*}
Hence in this case, \(\K^{(\lambda)} = \K^{(1)} + \int_1^{\lambda} \frac{d}{d \mu} \K^{(\mu)} d\mu\) converges in \(\D'(G;\End(V;E))\) as \(\lambda \to +\infty\), verifying claim (i). On the other hand, if \(k \geq Q\), for every \(i \in \N_0\),
\begin{equation*}
 q_i(x) \defeq \left. \frac{d^i}{d \varepsilon^i} \right|_{\varepsilon = 0} [\K' \circ \delta_{\varepsilon}(x)]
\end{equation*}
is an \(\End(V;E)\)--valued homogeneous polynomial of degree \(i\), and \(\K' \circ \delta_{\varepsilon}\) can be Taylor expanded at \(\varepsilon = 0\), leading to
\begin{equation*}
\K' \circ \delta_{\lambda^{-1}}(x) = \sum_{i=0}^{k-Q} q_i(x) \lambda^{-i} + \int_0^{\lambda^{-1}} \frac{(\lambda^{-1} - \varepsilon)^{k-Q}}{(k-Q)!}  \frac{d^{k-Q+1}}{d \varepsilon^{k-Q+1}} (\K' \circ \delta_{\varepsilon})(x) d\varepsilon
\end{equation*}
for \(\lambda > 0\). We then have
\begin{equation*}
\begin{split}
\frac{d}{d \lambda} \K^{(\lambda)}(x)
&= \lambda^{k-Q-1} \K' \circ \delta_{\lambda^{-1}}(x) \\
&= \sum_{i=0}^{k-Q} q_i(x) \lambda^{k-Q-i-1} + \lambda^{k-Q-1} \int_0^{\lambda^{-1}} \frac{(\lambda^{-1} - \varepsilon)^{k-Q}}{(k-Q)!}  \frac{d^{k-Q+1}}{d \varepsilon^{k-Q+1}} \K' \circ \delta_{\varepsilon}(x) d\varepsilon \\
&= \frac{d}{d \lambda} \left[ q_{k-Q}(x) \log \lambda + \sum_{i=0}^{k-Q-1} q_i(x) \frac{\lambda^{k-Q-i}}{k-Q-i} \right] + e^{(\lambda)}(x)
\end{split}
\end{equation*}
for some \(e^{(\lambda)}(x) \in C^{\infty}(G;\End(V;E))\); taking \(p_{k-Q} \defeq q_{k-Q}\) and \(p_{k-Q-i} \defeq \frac{q_{k-Q-i}}{i}\) for \(i = 1, \dotsc, k-Q\), and defining \(\tilde{\K}^{(\lambda)}\) by \eqref{eq:tildeKlam_def}, we have 
\begin{equation*}
\frac{d}{d \lambda} \tilde{\K}^{(\lambda)} = e^{(\lambda)}(x).
\end{equation*}
It remains to observe that for every compact subset \(K \subset G\), there exists a constant \(C_K\) (depending also on \(k\)) such that
\begin{equation*}
\sup_{\varepsilon > 0} \left\| \frac{d^{k-Q+1}}{d \varepsilon^{k-Q+1}} \K' \circ \delta_{\varepsilon}(x) \right\|_{L^{\infty}(K;\End(V;V))} \leq C_K.
\end{equation*}
As a result, for every \(\lambda > 0\) and every \(\phi \in \D(G;\End(V;E))\), we have
\begin{equation*}
|\langle e^{(\lambda)}, \phi \rangle| \leq C_{\text{supp}\, \phi} \|\phi\|_{L^1(G;\End(V;E))} \lambda^{k-Q-1} \int_0^{\lambda^{-1}} \frac{(\lambda^{-1} - \varepsilon)^{k-Q}}{(k-Q)!} d\varepsilon \lesssim \lambda^{-2}.
\end{equation*}
Integration in \(\lambda\) yields
\begin{equation*}
\begin{split}
 \int_1^{\infty} \Bigl| \Bigl\langle \frac{d}{d \lambda} \tilde{\K}^{(\lambda)}, \phi \Bigr\rangle \Bigr| d\lambda 
= \int_1^{\infty} \bigl| \bigl\langle e^{(\lambda)}, \phi \bigr\rangle \bigr| d\lambda 
< \infty
\end{split}
\end{equation*}
so that \(\tilde{\K}^{(\lambda)} = \tilde{\K}^{(1)} + \int_1^{\lambda} \frac{d}{d \mu} \tilde{\K}^{(\mu)} d\mu\) converges in \(\D'(G;\End(V;E))\) as \(\lambda \to +\infty\), verifying claim (ii).

Finally, to complete the proof of the second conclusion of the theorem, note that \(\K\) always agrees with some \(\End(V;E)\)--valued \(C^{\infty}\) function on \(G \setminus \{0\}\), by hypoellipticity of \(A^t(D)\). Suppose first \(k \geq Q\). Then 
\begin{equation*}
\K = \lim_{\lambda \to \infty} \tilde{\K}^{(\lambda)} = \lim_{\lambda \to \infty} \left[ \K^{(\lambda)} - (\log \lambda) p_{k-Q}  - \sum_{i=1}^{k-Q}  \lambda^{i} p_{k-Q-i} \right]
\end{equation*}
satisfies 
\begin{equation} \label{eq:Kcirc_scale}
\K \circ \delta_s = s^{k-Q} \bigl(\K + (\log s) P \bigr)
\end{equation}
for all \(s > 0\), where \(P \defeq - p_{k-Q}\). Indeed, recall \(\K^{(\lambda)} = \lambda^{k-Q} \K^{(1)} \circ \delta_{\lambda^{-1}}\), which gives, for any \(s > 0\), that \begin{equation*}\K^{(\lambda)} \circ \delta_s = s^{k-Q} \K^{(\lambda s^{-1})}.\end{equation*} 
It follows that
\begin{equation*}
\begin{split}
\tilde{\K}^{(\lambda)} \circ \delta_s 
&= s^{k-Q} \K^{(\lambda s^{-1})} - (\log \lambda) s^{k-Q} p_{k-Q}  - \sum_{i=1}^{k-Q}  \lambda^{i} s^{k-Q-i} p_{k-Q-i} \\
&= s^{k-Q} \biggl[ \K^{(\lambda s^{-1})} - (\log (\lambda s^{-1})) p_{k-Q}  - \sum_{i=1}^{k-Q}  (\lambda s^{-1})^{i} p_{k-Q-i} \biggr] - s^{k-Q} (\log s) p_{k-Q} \\
&= s^{k-Q} \tilde{\K}^{(\lambda s^{-1})} - s^{k-Q} (\log s) p_{k-Q}
\end{split}
\end{equation*}
for all \(s > 0\). Letting \(\lambda \to \infty\), we obtain \eqref{eq:Kcirc_scale}, which implies that \(\K - P(x) \log \|x\|\) is homogeneous of degree \(k-Q\). Since \(k-Q > -Q\), we see that \(\K - P(x) \log \|x\|\) is given by integration against some \(\End(V;E)\)--valued \(C^{\infty}\) function \(\K_{\infty}\) on \(G \setminus \{0\}\) that is homogeneous of degree \(k-Q\), as was to be proved. A similar but simpler calculation shows that if \(k < Q\), then \(\K\) is homogeneous of degree \(k-Q\), hence a kernel of type \(k\). This completes the proof of the theorem.
\end{proof}

\end{document}